\documentclass[11pt]{article}

\usepackage{amsmath,amsthm,amscd,latexsym}
\usepackage{amsfonts,pdfsync,color,graphicx}
\usepackage[psamsfonts]{amssymb}
\usepackage{epsfig}
\usepackage{color}
\usepackage[noadjust]{cite}
\usepackage{accents}

\usepackage{listings}
\usepackage{subfig}
\usepackage{multicol}

\newcommand{\R}{{\mathbb R}}

\usepackage[top=1.5in,bottom=1.5in,left=1.3in,right=1.0in]{geometry}

\newtheoremstyle{mystyle}               
  {}                
  {}                
  {}        
  {}                
  {\bfseries \itshape}       
  {.}      
  { }      
  {}       

\newtheorem{theorem}{Theorem}[section]

\newtheorem{lemma}[theorem]{Lemma} 
\newtheorem{corollary}[theorem]{Corollary}
\theoremstyle{definition}
\newtheorem{definition}[theorem]{Definition}
\newtheorem{example}[theorem]{Example}	
\theoremstyle{mystyle}
\newtheorem{remark}[theorem]{Remark}
\theoremstyle{proof}

\title{Relationship of the Green's functions related to the Hill's equation coupled to different  boundary value conditions$^{1,2,}$\footnote{Supported by Xunta de Galicia (Spain),  project EM2014/032 and Grant PID2020-113275GB-I00 funded by MCIN/AEI/10.13039/501100011033 and by “ERDF A way of making Europe” of the “European Union”.}}
\date{}
\author{Alberto Cabada, Luc{\' i}a L\'opez-Somoza and Mouhcine Yousfi\\
	$^1$CITMAga, 15782, Santiago de Compostela, Galicia, Spain\\
	$^2$Departamento de Estatística, Análise Matemática e Optimización\\
	Facultade de Matem\'aticas, Universidade de Santiago de Com\-pos\-te\-la, Spain.\\
	alberto.cabada@usc.es; lucia.lopez.somoza@usc.es; yousfi.mouhcine@usc.es}

\begin{document}
	\maketitle
	\begin{abstract}
		In this paper we will deduce  several properties  of the Green's functions related to the Hill's equation coupled to various boundary value conditions. In particular, the idea is to study the Green's functions of the second order differential operator coupled to Neumann, Dirichlet, Periodic and Mixed boundary conditions, by expressing the Green's function of a given problem as a linear combination of the Green's function of the other ones. This will allow us to compare different Green's functions when their sign is constant. Finally, such properties of the Green's function of the linear problem will be fundamental to deduce the existence of solutions to the nonlinear problem. The results are derived from the fixed point theory applied to related operators defined on suitable cones in Banach spaces.
	\end{abstract}
	
	
	\noindent {\bf MSC2020-Mathematics Subject Classification:}
	34B05, 34B08, 34B09, 34B15, 34B18, 34B27, 34B30\\

\section{Introduction}
This paper deals with the study of Green's functions related to Hill's equation
\[u''(t)+a(t)\,u(t)=0.\]
This equation has many applications in several fields as it models a large set of physical problems. Some examples of such applications are the inverted pendulum, Airy's equation or Mathieu's equation, which can be found in \cite{cabadacid,csizmadia,magnus,simmons,torres,zhang}.

Furthermore, it is important to note that the results obtained for Hill's equation can be easily extended (with a suitable change of variable, see \cite{simmons}) to a general second order linear differential equation of the form
\[u''(t)+a_1(t)\,u'(t)+a_0\,u(t)=0\]
provided that the functions $a_0$ and $a_1$ have enough regularity.

Moreover, the nonhomogeneous problem related to Hill's equation
\[u''(t)+a(t)\,u(t)=\sigma(t)\]
has also been extensively studied  (see \cite{akcay,cabadacid,A,cheng,hakl_torres,rodriguezcollado,torres,wang,yu,zhang,zhang2,zhang3} and the references therein), especially coupled to periodic conditions. In this sense, a particularly interesting case happens when $\sigma$ has constant sign, which can be interpreted as the action of an external force acting over the system on a certain direction (positive or negative). In such a case, the solutions of constant sign  of the equation can be interpreted as situations in which the deviation caused by the force is produced only in one direction (that is, the object oscillates only above or below the equilibrium point of the system).

It is in this context when the study of Green's functions gains importance, since the existence of solutions  of differential equations with constant sign  is directly related to the constant sign of the Green's functions. In particular, the fact that the Green's function related to a differential problem does not changes its sign allows the application of several topological and iterative methods to deduce existence results for suitable nonlinear problems.

Having this idea in mind, in \cite{kkkk} the authors develop a method which allows to write the Green's functions related to Neumann, Dirichlet and Mixed problems defined on the interval $[0,T]$ as a linear combination of Green's functions of some extended periodic problem (that is, the periodic problem was considered either on the interval $[0,2T]$ or on $[0,4T]$ and the potentials for these problems were the even extension $\widetilde{a}$ to $[0,2T]$ of the potential $a(t)$ considered on $[0,T]$ and the even extension of $\widetilde{a}$ to $[0,4T]$, respectively). As a consequence of such decomposition, the authors were able to deduce some comparison results between the solutions of the aforementioned problems. Moreover, they were able to relate the constant sign of the corresponding Green's functions.

This paper can be regarded then as a continuation of the work developed in \cite{kkkk} as our main objective will also be the decomposition of some Green's functions in terms of other ones. However, the techniques used in this paper are completely different to those mentioned for \cite{kkkk}. More concretely, we will consider two different ways of making the decomposition of Green's functions. The first one will be based on the superposition property of the solutions of a differential problem. On the other hand, the second one will make use of a general formula proved in \cite{Cabada}, which allows to relate two different Green's functions as long as the boundary value conditions of one of them can be rewritten in terms of the other one and both problems are nonresonant.

This way, we will consider periodic, Neumann, Dirichlet and mixed conditions and relate their corresponding Green's functions two by two. One of the differences between these approaches and the one considered in \cite{kkkk} is the fact that here we are able to find a relation between any pair of the aforementioned Green's functions, not only between any of them and the periodic one. Another difference is that in the present paper we are able to connect the  Green's function related to the periodic problem on $[0,T]$ with the Green's function related to any of the other cited boundary condition on $[0,T]$, which was not possible with the techniques used in \cite{kkkk}.  

As a consequence of the expressions relating the Green's functions, we are able to find some connections between their constant sign. Some of the results were already proved in \cite{kkkk} (although, the proof was different) and some others are, as far as we know, new in the literature.

The paper is divided into $5$ sections. In Section $2$ we compile some preliminary results from \cite{Cabada}. Sections $3$ and $4$ include the decomposition of Green's functions using the two different approaches mentioned before. Finally, Section $5$ includes an application to ensure the existence and find some bounds for the solution of nonlinear problems.

\section{Preliminaries}
Consider the second order linear operator  
\begin{equation*}
L\, u(t):=u''(t) +a(t) u(t),\;\; t\in I,
\end{equation*}
with $I\equiv [0,1]$, $a:I\rightarrow \mathbb{R}$, $a\in L^{1}(I)$, and 
\begin{equation*}
B_{i}(u):=\displaystyle \sum_{j=0}^{1} \left(\alpha_{j}^{i} u^{\left(j\right)}(a)+\beta_{j}^{i} u^{\left(j\right)}(b) \right),\quad i=1,\,2,
\end{equation*}
being $\alpha_{j}^{i},\;\; \beta_{j}^{i}$  real constants for $i=1,\,2,\;\; j=0,\,1$.

We will work on the space 
\begin{equation*}
W^{2,1}(I)=\{u\in C(I): u'\in AC(I)\},
\end{equation*} 
where  $AC(I)$ is the set of absolutely continuous functions on $I$. In particular, we will work with a Banach space $X\subset W^{2,1}(I)$ in which operator $L$ is non resonant, that is, the homogeneous equation 	\begin{equation*}
u''(t) +a(t) u(t)=0\;\; \text{a.e}.\;\;  t\in I,\quad u\in X,
\end{equation*} 
has as a unique solution the trivial one. In such a case, it occurs that for every $\sigma\in L^{1}(I)$ the non-homogeneous problem 
\begin{equation*}
u''(t) +a(t) u(t)=\sigma(t)\;\; \text{a.e.}\;\; t\in I,\quad  u\in X,
\end{equation*} 
 has a unique solution given by 
\begin{equation*}
 u(t)=\displaystyle \int_{0}^{1} G(t,s)\, \sigma(s)\,ds,\quad \forall t\in I,
\end{equation*}
where $G$ denotes the corresponding Green's function, which is the unique function that satisfies the following properties (see \cite{system} for details) 
 
\begin{definition}\label{d-Green-Function}
We say that $G\colon I\times I \rightarrow \mathbb{R}$ is a Green's function for problem 
\begin{equation*}
L\, u(t)=\sigma(t),\;\; \text{a.e}. \;\; t\in I, \quad 
B_1(u)=h_1,\quad B_2(u)=h_2,
\end{equation*}
being $\sigma\in L^{1}(I)$ and $h_1,\; h_2\in \R$, if it satisfies the following properties: 
\begin{itemize}
\item $G \in C(I \times I) \cap C^2((I \times I)\backslash{\{(s,s), s \in I\}})$.
\item For each $s\in (0,1)$, $G(\cdot,s)$ solves the differential equation $L y(t)=0$ on $[0,s)\cup(s,1]$ and satisfies the boundary conditions $B_1(G(\cdot,s))=B_2(G(\cdot,s))=0$.
\item For each $t\in (0,1)$ there exist the lateral limits $$\frac{\partial}{\partial t}G(t^{-},t)=\frac{\partial}{\partial t}G(t,t^{+}) \quad \text{and} \quad   \frac{\partial}{\partial t}G(t,t^{-})=\frac{\partial}{\partial t}G(t^{+},t)$$ and, moreover, $$\frac{\partial}{\partial t}G(t^{+}, t)-\frac{\partial }{\partial t}G(t^{-}, t)=\frac{\partial}{\partial t}G(t, t^{-})-\frac{\partial}{\partial t}G(t,t^{+})=1.$$
\end{itemize}

\end{definition}

We compile now some properties of Green's functions related to operator $L$. The following result is an adaptation of \cite[Lemma 1]{Cabada} to the problem considered in this paper.
\begin{lemma}\label{e-sol-hi}
	Problem 
	\begin{equation}\label{e-u}
	L\, u(t)=\sigma(t),\;\; \text{a.e.}\;\; t\in I, \quad  	B_1(u)=B_2(u)=0,
	\end{equation}
	has a unique Green's function if and only if the two following problems 
	\begin{equation*}
	L\,u(t)=0,\;\; \text{a.e.}\ t\in I, \quad
	B_1(u)=1,\quad 
	B_2(u)=0,
	\end{equation*}
	\begin{equation*}
	L\,u(t)=0,\;\; \text{a.e.}\ t\in I, \quad
	B_1(u)=0,\quad 
	B_2(u)=1,
	\end{equation*}
	have a unique solution that we denote as $\omega_{1}$ and $\omega_{2}$ respectively. 
	
	In such a case, for any $\sigma \in L^{1}(I)$, the following problem 
	\begin{equation*}
	L\,u(t)=\sigma(t),\;\; \text{a.e}\;\; t\in I,\quad
	B_1(u)=\lambda_1,\quad B_2(u)=\lambda_2,
	\end{equation*}
	has a unique solution given by
	\begin{equation*}
	u(t)=\displaystyle \int_{0}^{1} g(t,s)\, \sigma(s)\,ds+\lambda_1\, \omega_1(t) + \lambda_2\, \omega_2(t).
	\end{equation*}
\end{lemma}

Here, by considering $C_1,\, C_2:C^1(I)\rightarrow \mathbb{R}$, two linear and continuous operators, we formulate  the following result for general second order non-local boundary value problems. This result is an adaptation of \cite[Theorem 2]{Cabada} to the second order problem. The general result (which proves an analogous formula for the arbitrary $n$-th order problem) can be seen in \cite{Cabada}.
 
\begin{theorem}\label{th_Ci}
	Let us suppose that the homogeneous problem of \eqref{e-u} ($\sigma\equiv 0$) has a unique solution ($u\equiv 0$) and let $g$ be its related Green's function. 
	Let $\sigma \in L^{1}(I)$, and $\delta_1,\, \delta_2$ be such that
	\begin{equation*}
	\label{e-espectro}
	\det(I-A)\neq 0,
	\end{equation*}
	with $I$ the identity matrix of order $2$ and $A=(a_{ij})_{2\times 2}\in \mathcal{M}_{2\times 2}$ given by \[a_{ij}=\delta_{j}\,C_i(\omega_{j}), \quad i,\; j \in \{1, \, 2\}.\]
	Then problem
	\begin{equation}
		\label{e-linear-delta}	
	L\,u(t)=\sigma(t),\;\; \text{a.e.}\;\; t\in I,\quad
	B_1(u)=\delta_1\, C_1(u),\quad B_2(u)=\delta_2\, C_2(u),
	\end{equation}
	 has a unique solution $u\in C^2(I)$, given by the expression
	\begin{equation*}
	u(t)=\displaystyle \int_{0}^{1} G(t,s,\delta_{1},\delta_2) \sigma(s) ds,
	\end{equation*}
	where 
	\begin{equation}\label{e-formula}
	G(t,s,\delta_{1},\delta_2):=g(t,s)+ \sum_{i=1}^{2} \sum_{j=1}^{2} \delta_{i} \, b_{ij} \, \omega_{i}(t)  \, C_j(g(\cdot,s)),\quad t, \; s \in I,
	\end{equation}
	with $B=(b_{ij})_{2\times 2}=(I-A)^{-1}$.
\end{theorem}
 
 For any $\lambda \in \mathbb{R}$, consider operator $L[\lambda]$ defined  as follows 
 \begin{equation*}
 L[\lambda]\, u(t)\equiv u''(t)+(a(t)+\lambda)\, u(t),\;\;   t\in I.
 \end{equation*}
  When working with this operator, to stress the dependence of the Green's function  on the parameter $\lambda$, we will denote by $G[\lambda]$ the Green's function related to $L[\lambda]$.

  
  In this paper, we will deal with some problems related to operator $L[\lambda]$, which we describe in the sequel:
  \begin{itemize}
  \item Neumann problem:
  \begin{equation}\label{e-Neumann}
 L[\lambda]\,u(t)=\sigma(t),\;\; \text{a.e.}\;\; t\in I,\;\; u\in X_{N}=\{u\in W^{2,1}(I): u'(0)=u'(1)=0 \}.
 \end{equation}
  \item Dirichlet problem:
   \begin{equation}\label{e-Dirichlet}
  L[\lambda]\,u(t)=\sigma(t),\;\; \text{a.e.}\;\; t\in I,\;\; u\in X_{D}=\{u\in W^{2,1}(I): u(0)=u(1)=0 \}.
  \end{equation}
  \item Mixed problem 1:
  \begin{equation}\label{e-Mixed-1}
  L[\lambda]\,u(t)=\sigma(t),\;\; \text{a.e.}\;\; t\in I,\;\; u\in X_{M_{1}}=\{u\in W^{2,1}(I): u'(0)=u(1)=0 \}.
  \end{equation}
  \item Mixed problem 2:
  \begin{equation}\label{e-Mixed-2}
 L[\lambda]\,u(t)=\sigma(t),\:\; \text{a.e.} \;\; t\in I,\; u \in X_{M_{2}}=\{u\in W^{2,1}(I): u(0)=u'(1)=0 \}.
  \end{equation}
   \item Periodic problem:
  \begin{equation}\label{e-periodic}
  L[\lambda]\,u(t)\,=\sigma(t),\;\; \text{a.e.}\;\; t\in I,\;\; u\in X_{P}=\{u\in W^{2,1}(I): u(0)=u(1),\; u'(0)=u'(1) \}.
  \end{equation}
 \end{itemize}
 We denote by $G_{D}[\lambda]$, $G_{P}[\lambda]$, $G_{N}[\lambda]$, $G_{M_{1}}[\lambda]$ and $G_{M_{2}}[\lambda]$ the Green's function related to Dirichlet, Periodic, Neumann, Mixed 1 and Mixed 2 problems, respectively. Moreover, we denote by $u_{D}$, $u_{P}$, $u_{N}$, $u_{M_{1}}$ and $u_{M_{2}}$  the solutions of the corresponding problems and by $\lambda_{0}^{D}$, $\lambda_{0}^{P}$, $\lambda_{0}^{N}$, $\lambda_{0}^{M_{1}}$ and $\lambda_{0}^{M_{2}}$ the first eigenvalues of each problem.

Now, let us consider the following first order differential $2$-dimensional linear system 
 \begin{equation}\label{af}
x'(t)=A(t)x(t)+f(t),\;\; \text{a.e} \;\; t\in I,
 \end{equation} 
subject to the two-point boundary value condition 
\begin{equation}\label{bc}
B\,x(0)+C\,x(1)=0,
\end{equation} 
being $A \in L^{1}(I,M_{2\times 2})$,  $f\in L^{1}(I,\mathbb{R}^{2})$, $B,C \in \mathcal{M}_{2\times 2}$, and $x\in AC(I,\mathbb{R}^{2})$. 

From \cite[page 22]{system}, we know that the Green's function related to \eqref{af}-\eqref{bc}, denoted by $g$,  satisfies that
\begin{equation}\label{referencia1}
B\, g(0,0) +C\, g(1,0)=B
\end{equation}
and 
\begin{equation}\label{referencia2}
B\, g(0,1)+C\, g(1,1)=-C.
\end{equation}
Now, we observe that the equation
\begin{equation}\label{ecuacion2}
 L[\lambda]\,u(t)= \sigma(t),\;\; \text{a.e.}\;\; t\in I
\end{equation}
 can be rewritten as a system of type  \eqref{af} as follows
\begin{equation}\label{sistema1}
\begin{pmatrix}
u(t)\\
u'(t)
\end{pmatrix}'=\begin{pmatrix}
0 & 1\\
-a(t)-\lambda & 0
\end{pmatrix}
\begin{pmatrix}
u(t)\\
u'(t)
\end{pmatrix}
+\begin{pmatrix}
0\\
\sigma(t)
\end{pmatrix}.
\end{equation}
In this case, we have that 
\begin{equation*}
A(t)=\begin{pmatrix}
0 & 1\\
-a(t)-\lambda & 0
\end{pmatrix}\;\;\text{and}\;\; f(t)=\begin{pmatrix}
0\\
\sigma(t)
\end{pmatrix}.
\end{equation*}
Now, we will give the expression of the different problems related to operator $L[\lambda]$ mentioned above based on equation \eqref{bc}, by giving the corresponding matrices $B$ and $C$ of each case:
\begin{itemize}
	\item Neumann problem:
 $$B=\begin{pmatrix}
	0 & 1\\
	0 & 0
	\end{pmatrix} \quad \text{and} \quad C=\begin{pmatrix}
	0 & 0\\
	0 & 1
	\end{pmatrix}
	.$$
	\item Dirichlet problem:
$$B=\begin{pmatrix}
    1 & 0\\
    0 & 0
    \end{pmatrix}
    \quad \text{and} \quad C=\begin{pmatrix}
    0 & 0\\
    1 & 0
    \end{pmatrix}
    .$$
	\item Mixed problem 1:
 $$B=\begin{pmatrix}
	0 & 1\\
	0 & 0
	\end{pmatrix}
	\quad \text{and} \quad  C=\begin{pmatrix}
	0 & 0\\
	1 & 0
	\end{pmatrix}
.$$
	\item Mixed problem 2:
$$B=\begin{pmatrix}
	1 & 0\\
	0 & 0
	\end{pmatrix}
	\quad \text{and} \quad C=\begin{pmatrix}
	0 & 0\\
	0 & 1
	\end{pmatrix}
	.$$
	\item Periodic problem:
 $$B=\begin{pmatrix}
	1 & 0\\
	0 & 1
	\end{pmatrix}
	\quad \text{and} \quad  C=\begin{pmatrix}
	-1 & 0\\
	0 & -1
	\end{pmatrix}.
	$$
\end{itemize}
\begin{remark}
The matrices $B$ and $C$ are not unique since we can take as $B$ and $C$ a multiple $k\,B$ and $k\,C$ with $k$ a nonzero real number. We can also swap the rows of the two matrices $B$ and $C$.	
\end{remark}
Using \cite[page 11]{cl} we know that the matrix function 
\begin{equation*}
g[\lambda](t,s)=\begin{pmatrix}
-\frac{\partial}{\partial s} G[\lambda](t,s) & G[\lambda](t,s)\\
-\frac{\partial^{2} }{\partial s\partial t}G[\lambda](t,s) & \frac{\partial}{\partial t}G[\lambda](t,s)
\end{pmatrix}
\end{equation*}
 is the Green's function related to  system \eqref{sistema1} associated with the differential equation \eqref{ecuacion2}, coupled to  the boundary conditions \eqref{bc} where $G[\lambda]$ is the Green's function of the linear equation \eqref{ecuacion2} coupled to boundary conditions  \eqref{bc} under the notation $x=\begin{pmatrix}
 u\\
 u'
 \end{pmatrix}$.

 Now we introduce some auxiliary functions that we are going to use throughout this article to relate the different problems that we have defined above. 
 
 Let us define $r_{1}[\lambda]$ as the unique solution to the problem 
  \begin{equation}\label{5}
 L[\lambda]\, u(t)=0,\;\; \text{a.e.}\;\; t\in I,\;\; u(0)=1,\;\; u(1)=0,
\end{equation}

  $r_{2}[\lambda]$ as the unique solution to
 \begin{equation}\label{4}
 L[\lambda]\, u(t)=0,\;\; \text{a.e}\;\; t\in I,\;\; u(0)=0,\;\; u(1)=1,
\end{equation}
 
 $r_{3}[\lambda]$ as the unique solution to
 \begin{equation}\label{39}
 L[\lambda]\, u(t)=0,\;\; \text{a.e.}\;\; t\in I,\;\; u(0)-u(1)=1,\;\; u'(0)-u'(1)=0,
 \end{equation}
 
 $r_{4}[\lambda]$ as the unique solution to
  \begin{equation}\label{19}
 L[\lambda]\, u(t)=0,\;\; \text{a.e.}\;\; t\in I,\;\; u(0)-u(1)=0,\;\; u'(0)-u'(1)=1,
 \end{equation}
 
 $r_{5}[\lambda]$ as the unique solution to
 \begin{equation*}
 L[\lambda]\, u(t)=0,\;\;  \text{a.e.}\;\; t\in I,\;\; u'(0)=1,\;\;  u'(1)=0,
 \end{equation*}
 
 $r_{6}[\lambda]$ as the unique solution to
 \begin{equation*}
 L[\lambda]\, u(t)=0,\;\; \text{a.e.}\;\;  t\in I,\;\; u'(0)=0,\;\; u'(1)=1,
 \end{equation*}
 
 $r_{7}[\lambda]$ as the unique solution to  
 \begin{equation*}
 L[\lambda]\, u(t)=0,\;\;  \text{a.e.}\;\; t\in I,\;\; u(0)=1,\;\; u'(1)=0,
 \end{equation*}
 
  $r_{8}[\lambda]$ as the unique  solution to  
 \begin{equation*}
 L[\lambda]\, u(t)=0,\;\;  \text{a.e.}\;\; t\in I,\;\; u(0)=0,\;\; u'(1)=1,
 \end{equation*}
 
   $r_{9}[\lambda]$ as the unique  solution to  
  \begin{equation*}
  L[\lambda]\, u(t)=0,\;\; \text{a.e.}\;\; t\in I,\;\;  u'(0)=1,\;\; u(1)=0,
  \end{equation*}
  
  and $r_{10}[\lambda]$ as the unique solution of the problem 
  \begin{equation*}
  L[\lambda]\, u(t)=0,\;\; \text{a.e.}\;\; t\in I,\;\; u'(0)=0,\;\; u(1)=1.
  \end{equation*}
  Now, we will find the expression of $r_{1}[\lambda]$ as a function of the Green's function of the Dirichlet problem  using the equalities \eqref{referencia1} and \eqref{referencia2}. 
  
  For the Dirichlet problem, equation \eqref{referencia2} becomes the following equality
\begin{equation*}
\footnotesize
\begin{pmatrix}
1 & 0\\
0 & 0
\end{pmatrix}
\begin{pmatrix}
-\frac{\partial}{\partial s} G_{D}[\lambda](0,0) & G_{D}[\lambda](0,0)\\
-\frac{\partial^{2} }{\partial s\partial t}G_{D}[\lambda](0,0) & \frac{\partial}{\partial t} G_{D}[\lambda](0,0)
\end{pmatrix}
+\begin{pmatrix}
0 & 0\\
1 & 0
\end{pmatrix}
\begin{pmatrix}
-\frac{\partial }{\partial s}G_{D}[\lambda](1,0) & G_{D}[\lambda](1,0)\\
-\frac{\partial^{2} }{\partial s\partial t}G_{D}[\lambda](1,0) & \frac{\partial }{\partial t}G_{D}[\lambda](1,0)
\end{pmatrix}=\begin{pmatrix}
1 & 0\\
0 & 0
\end{pmatrix}.
\end{equation*}
Therefore, 
\begin{equation*}
\begin{matrix}
-\frac{\partial }{\partial s}G_{D}[\lambda](0,0)=1, & \frac{\partial }{\partial s}G_{D}[\lambda](1,0)=0,\\
 G_{D}[\lambda](0,0)=0,&  G_{D}[\lambda](1,0)=0.
\end{matrix}
\end{equation*}

 By the uniqueness of the function $r_{1}[\lambda]$, it follows that 
\begin{equation*}
r_{1}[\lambda](t)=-\frac{\partial }{\partial s}G_{D}[\lambda](t,0).
\end{equation*}
Making similar arguments, we can deduce that
\begin{equation*}
\begin{array}{lll}
r_{2}[\lambda](t)=\frac{\partial }{\partial s}G_{D}[\lambda](t,1), \qquad & r_{3}[\lambda](t)=-\frac{\partial }{\partial s}G_{P}[\lambda](t,0), \qquad & r_{4}[\lambda](t)=G_{P}[\lambda](t,0),\\
r_{5}[\lambda](t)=G_{N}[\lambda](t,0),& r_{6}[\lambda](t)=-G_{N}[\lambda](t,1), & r_{7}[\lambda](t)=-\frac{\partial }{\partial s}G_{M_{2}}[\lambda](t,0),\\
 r_{8}[\lambda](t)=-G_{M_{2}}[\lambda](t,1), & r_{9}[\lambda](t)=G_{M_{1}}[\lambda](t,0),& r_{10}[\lambda](t)=- \frac{\partial}{\partial s} G_{M_{1}}[\lambda](t,1).\\
\end{array}
\end{equation*}
 \section{Decomposing Green's functions}
 This section is devoted to the study of the relationships between the expressions of  the Green's  functions related to problems \eqref{e-Neumann}, \eqref{e-Dirichlet}, \eqref{e-Mixed-1}, \eqref{e-Mixed-2} and \eqref{e-periodic}.
 
 To do this end, we will compare the different expressions by putting each boundary condition as a combination of the others. 
 
 Such expressions will be deduced from Lemma  \ref{e-sol-hi}. We pay attention to the fact that in this case we are considering the potential $a(t)$ and the definition on the interval $[0,1]$.  So, we make a different approach to the one given in \cite{kkkk} where the expressions are obtained for the corresponding extensions of the potential $a(t)$ to the intervals $[0,2]$ and $[0,4]$.  
 \subsection{Dirichlet and Periodic problems} 
In this subsection we study the relation between the Green's functions of Dirichlet and Periodic problems.
\begin{theorem}\label{th3}
 If operator $L[\lambda]$ is nonresonant both in $X_{D}$ and $X_{P}$, then it holds that 
 \begin{equation}\label{e-Gp-Gd}
 \small
 \begin{aligned}
 G_{P}[\lambda](t,s)&=G_{D}[\lambda](t,s)-\left(r_{1}[\lambda](t)+r_{2}[\lambda](t)\right)\, G_{P}[\lambda](1,s)\\
 &=G_{D}[\lambda](t,s)+\left(\frac{\partial }{\partial s}G_{D}[\lambda](t,1)-\frac{\partial }{\partial s}G_{D}[\lambda](t,0)\right)\, G_{P}[\lambda](1,s),\, \forall (t,s)\in I\times I.
 \end{aligned}
 \end{equation}
\end{theorem}
\begin{proof}
We express the Periodic problem \eqref{e-periodic} in function of the Dirichlet one \eqref{e-Dirichlet} as follows
\begin{equation}\label{12}
L[\lambda]\, u(t)=\sigma(t),\;\; \text{a.e.}\;\; t\in I,\;\; u(0)=u(1),\;\; u(1)=u(1)+u'(0)-u'(1).
\end{equation}
Then, using Lemma \ref{e-sol-hi}, we have that the solution of problem \eqref{12} is given by the following expression 
\begin{equation*}
\begin{aligned}
u_{P}(t)=&\displaystyle \int_{0}^{1} G_{P}[\lambda](t,s)\, \sigma(s)\, ds\\
=&\displaystyle \int_{0}^{1} G_{D}[\lambda](t,s)\, \sigma(s) \, ds+r_{1}[\lambda](t) u_{P}(1)+r_{2}[\lambda](t) \left(u_{P}(1)+u'_{P}(0)-u'_{P}(1)\right)\\
=&\displaystyle \int_{0}^{1} G_{D}[\lambda](t,s)\, \sigma(s) \,ds+r_{1}[\lambda](t) \displaystyle \int_{0}^{1} G_{P}[\lambda](1,s)\, \sigma(s) \,ds \\
&+r_{2}[\lambda](t) \displaystyle \int_{0}^{1} \left[G_{P}[\lambda](1,s)+ \frac{\partial }{\partial t}G_{P}[\lambda](0,s)-\frac{\partial }{\partial t} G_{P}[\lambda](1,s)\right]\, \sigma(s)\, ds\\
=&\displaystyle \int_{0}^{1}\left[G_{D}[\lambda](t,s)+\left(r_{1}[\lambda](t)+r_{2}[\lambda](t)\right)G_{P}[\lambda](1,s)\right]\,\sigma(s)\,ds,
\end{aligned}
\end{equation*} 
where the last equality follows from Definition \ref{d-Green-Function}, condition $(G6)$: $\frac{\partial}{\partial t} G_{P}[\lambda](0,s)=\frac{\partial }{\partial t}G_{P}[\lambda](1,s)$, $\forall s\in (0,1)$.
 
Since previous equalities hold for every $\sigma \in L^{1}(I)$, we obtain \eqref{e-Gp-Gd}.
 \end{proof} 
\begin{remark}
	We point out that, as a direct consequence of Lemma \ref{e-sol-hi}, we have that both $r_{1}[\lambda]$ and $r_{2}[\lambda]$ are uniquely determined. In fact, with  the notation used in Lemma \ref{e-sol-hi}, we have that $B_{1}(u)=u(0)$, $B_{2}(u)=u(1)$, $\omega_{1}=r_{1}[\lambda]$ and $\omega_{2}=r_{2}[\lambda]$.
\end{remark}

 Next, we study  the oscillation of the  functions $r_{1}[\lambda]$ and $r_{2}[\lambda]$ using the Sturm-Liouville theory of eigenvalues. 
 Let $\{\lambda_{n}^{D}\}_{n=0}^{\infty}$  be the sequence of eigenvalues of the Dirichlet problem:
 \begin{equation*}
 \left(D_{\lambda}\right)\quad 
 L[\lambda]\, u(t)=0,\;\; \text{a.e.}\;\; t\in I,\;\; u(0)=u(1)=0.
 \end{equation*}
  It is very well-known that $\lim\limits_{n\to \infty} \lambda_{n}^{D}=\infty$ (see \cite[Theorem 4.3.1]{Zetel}), and that any of the eigenvalues has a single associated eigenvector $v_{n}$, such that 
  \begin{equation*}
  \left(D_{n}\right)\quad 
  L[\lambda_{n}^{D}]\, v_{n}(t)=0,\;\; \text{a.e.}\;\; t\in I,\;\;
  v_{n}(0)=v_{n}(1)=0,
  \end{equation*} 
  with exactly $n$ zeros in $(0,1)$.
  
  Moreover, this eigenfunction satisfies that  $v'_{n}(0)\neq 0$. 
\begin{lemma}\label{autovalor2}
	Problem \eqref{5} has a unique solution if and only if $\lambda\neq \lambda_{n}^{D}$, $n=0,1,\ldots$.
\end{lemma}
\begin{lemma}\label{signo2}
	The unique solution $r_{1}[\lambda]$ of problem \eqref{5} has exactly $n$ zeros in $(0,1)$ if and only if $\lambda\in (\lambda_{n-1}^{D},\lambda_{n}^{D})$, $n=1,2,\ldots,$ and $r_{1}[\lambda]>0$ on $[0,1)$ if and only if $\lambda<\lambda_{0}^{D}$.
	In addition, $\left(-1\right)^{n} r'_{1}(1)<0$ for all $\lambda\in (\lambda_{n-1}^{D},\lambda_{n}^{D})$, $n=1,2,\ldots,$ and $r'_{1}[\lambda](1)<0$, for all $\lambda<\lambda_{0}^{D}$.
\end{lemma}

\begin{lemma}\label{lema1}
	Problem \eqref{4} has a unique solution if and only if $\lambda\neq \lambda_{n}^{D}$, $n=0,1,\ldots$.
\end{lemma}
\begin{lemma}\label{signo1}
	The unique solution  of problem \eqref{4}  $r_{2}[\lambda]$ has exactly $n$ zeros in $\left(0,1\right)$ if and only if $\lambda\in (\lambda_{n-1}^{D},\lambda_{n}^{D})$, $n=1,2,\ldots$ and $r_{2}[\lambda]>0$ on $(0,1]$ if and only if $\lambda<\lambda_{0}^{D}$.
	In addition, $\left(-1\right)^{n} r'_{2}[\lambda](0)>0$ for all $\lambda\in (\lambda_{n-1}^{D},\lambda_{n}^{D})$, $n=1,2,\ldots$ and $r'_{2}[\lambda](0)>0$, for all $\lambda<\lambda_{0}^{D}$.
\end{lemma}
\begin{remark}
Lemmas \ref{autovalor2} and \ref{lema1} are a corollary of Lemma \ref{e-sol-hi}. Theorems \ref{signo1} and \ref{signo2} follow from Sturm's comparison theorem.	
\end{remark}
As a direct consequence of equality \eqref{e-Gp-Gd}, we deduce the following comparison between the values of the Green's functions related to the Dirichlet and Periodic problems.
\begin{theorem}\label{teorema1}
 The following inequality holds 
\begin{equation}\label{e-des-Gp-Gd}
G_{P}[\lambda](t,s)<G_{D}[\lambda](t,s)<0,\;\; \forall (t,s) \in  (0,1)\times (0,1),\;\;\forall \lambda<\lambda_{0}^{P}.
\end{equation}
\end{theorem}
\begin{proof} 
 It is immediate to verify  that the  function $r[\lambda](t):=r_{1}[\lambda](t)+r_{2}[\lambda](t)$ solves the following problem 
 \begin{equation*}
 L[\lambda]\, r[\lambda](t)=0,\;\; \text{a.e.} \;\;t\in I,\;\; r[\lambda](0)=r[\lambda](1)=1.
 \end{equation*}
From Lemmas \ref{signo2} and \ref{signo1}, it is obvious that if $\lambda<\lambda_{0}^{D}$ then $r[\lambda](t)>0$ for all $t\in I$.

Moreover, we know  that $G_{P}[\lambda]$ is negative on $I\times I$ for all $\lambda<\lambda_{0}^{P}$ and $G_{D}[\lambda]$ is negative on $(0,1)\times (0,1)$ for all $\lambda<\lambda_{0}^{D}$ (see (\cite[Lemma 2.9]{kkkk}). In addition, $\lambda_{0}^{P} < \lambda_{0}^{D}$ (\cite[page 44]{A}).

As $r[\lambda](t)>0$ for all $t\in I$ when $\lambda<\lambda_{0}^{D}$, using \eqref{e-Gp-Gd}, we obtain the result.
\end{proof}
\begin{remark}
 From equality \eqref{e-Gp-Gd} it follows that 
 \begin{equation*}
 \dfrac{G_{D}[\lambda](t,s)-G_{P}[\lambda](t,s)}{G_{P}[\lambda](1,s)}=-r[\lambda](t),
 \end{equation*}
 if $\lambda$ is not a eigenvalue of the Dirichlet and Periodic problems.
 
  Deriving the above equality with respect to $s$, we obtain the following identity
 \begin{equation*}
 \footnotesize
 \frac{\partial}{\partial s} \left(G_{D}[\lambda](t,s)-G_{P}[\lambda](t,s)\right) G_{P}[\lambda](1,s)=\left(G_{D}[\lambda](t,s)-G_{P}[\lambda](t,s)\right) \frac{\partial }{\partial s} G_{P}[\lambda](1,s),\;\; \forall (t,s)\in I\times I.
 \end{equation*}
\end{remark}

In the sequel we will do an alternative study to the one done in Theorem  \ref{th3}. In this case by considering the Dirichlet conditions as a combination of the periodic ones.

 Let us write the Dirichlet problem as a function of the periodic problem as follows 
 \begin{equation*}
 L[\lambda]\, u(t)=\sigma(t),\;\;\text{a.e.}\;\; t\in I,\;\; u(0)-u(1)=-u(1),\;\; u'(0)-u'(1)=u'(0)-u'(1)+u(1).
 \end{equation*}
 Taking into account that  $r_{4}[\lambda](t)=G_{P}[\lambda](t,0)$ solves \eqref{19}, performing the calculations in an analogous way than before, using Lemma \ref{e-sol-hi} the following result is attained.
 \begin{theorem}
 Assume that operator $L[\lambda]$ is nonresonant both in $X_{D}$ and $X_{P}$, then it holds that  
 \begin{equation}\label{e-GD-GP}
 \small
 \begin{aligned}
 G_{D}[\lambda](t,s)=&G_{P}[\lambda](t,s)+r_{4}[\lambda](t)\, \left(\frac{\partial}{\partial t} G_{D}[\lambda](0,s)-\frac{\partial }{\partial t}G_{D}[\lambda](1,s)\right)\\
 =&G_{P}[\lambda](t,s)+G_{P}[\lambda](t,0)\,  \left(\frac{\partial }{\partial t}G_{D}[\lambda](0,s)-\frac{\partial }{\partial t}G_{D}[\lambda](1,s)\right),\;\; \forall (t,s)\in I\times I.
 \end{aligned}
 \end{equation}
 \end{theorem}
 \begin{remark}
 	Notice that, if $\lambda<\lambda_{0}^{D}$ we have that $G_{D}[\lambda]<0$ on $(0,1)\times (0,1)$ and, as a consequence  $$\frac{\partial }{\partial t} G_{D}[\lambda](0,s)<0<\frac{\partial }{\partial t} G_{D}[\lambda](1,s),\quad s\in (0,1).$$
 	
 	Moreover, if $\lambda<\lambda_{0}^{D}$ then $G_{P}[\lambda]<0$ on $I\times I$. So, from \eqref{e-GD-GP} and the fact that $\lambda_{0}^{P}<\lambda_{0}^{D}$ we deduce the inequality \eqref{e-des-Gp-Gd} again.
 \end{remark}
\subsection{Dirichlet and Neumann problems}
In this section we continue the work done in previous section. In this case we will consider the Dirichlet and Neumann problems. We will obtain some expressions  that allow us to connect both Green's functions. 
  
\begin{theorem}
Assume that operator $L[\lambda]$ is nonresonant in the spaces $X_{D}$ and $X_{N}$. Then the  following equality is satisfied
\begin{equation}\label{6}
\small
\begin{aligned}
G_{N}[\lambda](t,s)&=G_{D}[\lambda](t,s)+r_{1}[\lambda](t)\, G_{N}[\lambda](0,s)+r_{2}[\lambda](t)\, G_{N}[\lambda](1,s)\\
&=G_{D}[\lambda](t,s)-\frac{\partial }{\partial s}G_{D}[\lambda](t,0)\, G_{N}[\lambda](0,s)+\frac{\partial }{\partial s}G_{D}[\lambda](t,1)\, G_{N}[\lambda](1,s),\;\; \forall (t,s) \in I\times I.
\end{aligned}
\end{equation}
\end{theorem}
\begin{proof}
 Let us rewrite Neumann problem in the following way 
\begin{equation*}
L[\lambda]\, u(t)=\sigma(t),\;\;  \text{a.e.}\;\; t\in I,\;\; u(0)=u(0)+u'(0),\;\; u(1)=u(1)+u'(1).
\end{equation*}

Using Lemma \ref{e-sol-hi}, the solution to the above  problem is 
\begin{equation*}
\begin{aligned}
u_{N}(t)=&\displaystyle \int_{0}^{1} G_{N}[\lambda](t,s)\, \sigma(s)\, ds
=\displaystyle \int_{0}^{1} G_{D}[\lambda](t,s)\, \sigma(s)\, ds+ r_{1}[\lambda](t)  u_{N}(0)+r_{2}[\lambda](t) u_{N}(1)\\
=& \displaystyle \int_{0}^{1} G_{D}[\lambda](t,s)\, \sigma(s) ds+ r_{1}[\lambda](t) \displaystyle \int_{0}^{1} G_{N}[\lambda](0,s)\, \sigma(s) ds +r_{2}[\lambda](t) \displaystyle \int_{0}^{1} G_{N}[\lambda](0,s)\, \sigma(s) ds.
\end{aligned}
\end{equation*}
Therefore, since previous equalities hold for every  $\sigma\in L^{1}(I)$, we obtain \eqref{6}.
\end{proof}
\begin{corollary}
	The following inequality holds 
\begin{equation}\label{e-GN-GD}
G_{N}[\lambda](t,s)<G_{D}[\lambda](t,s)<0,\quad \forall (t,s)\in (0,1)\times (0,1),\quad \forall \lambda<\lambda_{0}^{N}. 
\end{equation}
\end{corollary}
\begin{proof}
We know that, from Lemmas \ref{signo2} and \ref{signo1}, $r_{1}[\lambda]$ and $r_{2}[\lambda]$ are positive on $(0,1)$ for all  $\lambda<\lambda_{0}^{D}$. In addition, $\lambda_{0}^{N}\leq \lambda_{0}^{P}< \lambda_{0}^{D}$ (\cite[page 44]{A}), $G_{N}[\lambda]<0$ on $I\times I$ for all $\lambda<\lambda_{0}^{N}$ (\cite[Corollary 4.5]{kkkk}) and $G_{D}[\lambda]<0$ on $(0,1)\times(0,1)$ for all $\lambda<\lambda_{0}^{D}$ (\cite[Lemma 2.9]{kkkk}).  Then,  for all $\lambda<\lambda_{0}^{N}$, $r_{1}[\lambda]$ and $r_{2}[\lambda]$ are positive on $(0,1)$. Hence, using \eqref{6} we obtain the result.
\end{proof}
\begin{remark}
The above result can be deduced from  \cite[Corollaries 4.5, 4.8 and 4.10]{kkkk}, but in a different way than we have explained here. In such reference, the argument used is based on considering  the even extension of the solution to the interval $[0,2]$. In any case, the expression \eqref{6} relating $G_{N}[\lambda]$ and $G_{D}[\lambda]$ is different from the one obtained in that article.
\end{remark} 
For the reverse process, by writing the Dirichlet problem as a function of Neumann problem as 
\begin{equation*}
L[\lambda]\, u(t)=\sigma(t),\;\;  \text{a.e.}\;\; t\in I,\;\; u'(0)=u(0)+u'(0),\;\; u'(1)=u(1)+u'(1),
\end{equation*} 
we arrive to  the next result as a consequence of Lemma \ref{e-sol-hi}. 
 \begin{theorem}
 	Assume that operator $L[\lambda]$ is nonresonant in the spaces $X_{D}$ and $X_{N}$. Then, the following equalities is satisfied
 \begin{equation}\label{e-GD-GN}
  \small
   \begin{aligned}
 G_{D}[\lambda](t,s)&=G_{N}[\lambda](t,s)+r_{5}[\lambda](t)\, \frac{\partial }{\partial t}G_{D}[\lambda](0,s)+r_{6}[\lambda](t)\, \frac{\partial }{\partial t}G_{D}[\lambda](1,s)\\
 &=G_{N}[\lambda](t,s)+G_{N}[\lambda](t,0)\, \frac{\partial }{\partial t}G_{D}[\lambda](0,s)-G_{N}[\lambda](t,1)\, \frac{\partial }{\partial t}G_{D}[\lambda](1,s) ,\;\; \forall (t,s) \in I\times I. 
  \end{aligned}
 \end{equation}
 \end{theorem}
\begin{remark}
	Since for $\lambda<\lambda_{0}^{N}$ we have $G_{N}[\lambda]<0$ on $I\times I$ and $G_{D}[\lambda]<0$ on $(0,1)\times (0,1)$, we conclude from \eqref{e-GD-GN} the inequality \eqref{e-GN-GD} again.
\end{remark}
\subsection{Dirichlet and Mixed problems}
In this case we carry out an analysis of the relationship between the Green's functions of Dirichlet and Mixed problems. Following the same steps than before in the previous subsection, we get the next result.


\begin{theorem}
 Assume that $L[\lambda]$ is nonresonant  both in $X_{D}$ and $X_{M_{1}}$, then it holds that  
 \begin{equation}\label{7}
 \begin{aligned}
 G_{M_{1}}[\lambda](t,s)&=G_{D}[\lambda](t,s)+r_{1}[\lambda](t)\, G_{M_{1}}[\lambda](0,s)\\
 &=G_{D}[\lambda](t,s)-\frac{\partial }{\partial s}G_{D}[\lambda](t,0)\, G_{M_{1}}[\lambda](0,s),\;\; \forall (t,s)\in I\times I.
 \end{aligned}
 \end{equation}
 \end{theorem}


As a consequence, we deduce the following result.
\begin{corollary}\label{corolario3}
The following inequality holds 
	\begin{equation}\label{e-des-GM1-GD}
	G_{M_{1}}[\lambda](t,s)<G_{D}[\lambda](t,s)<0,\;\; \forall (t,s)\in (0,1)\times(0,1),\;\; \forall  \lambda<\lambda_{0}^{M_{1}}.
	\end{equation}
\end{corollary}
\begin{proof}
Inequality $\lambda_{0}^{M_{1}}<\lambda_{0}^{D}$ is provided in \cite[Remark 4.19]{kkkk}. In addition,  we have that $G_{M_{1}}[\lambda]<0$ on $[0,1)\times [0,1)$ if and only if $\lambda<\lambda_{0}^{M_{1}}$ (see \cite[Corollary 4.7]{kkkk}) and $G_{D}[\lambda]<0$ on $(0,1)\times (0,1)$ if and only if $\lambda<\lambda_{0}^{D}$, which implies that $\frac{\partial}{\partial s}G_{D}(t,0)<0$ for all $\lambda<\lambda_{0}^{D}$ and $t\in (0,1)$. 

Therefore, using \eqref{7} we deduce the  inequality.
\end{proof}
Similarly, for Mixed 2 problem we arrive  at the following results.
\begin{theorem}
If operator $L[\lambda]$ is nonresonant in $X_{D}$ and $X_{M_{2}}$, then the following equality holds 
\begin{equation}\label{e-M2-D}
\begin{aligned}
G_{M_{2}}[\lambda](t,s)&=G_{D}[\lambda](t,s)+r_{2}[\lambda](t)\, G_{M_{2}}[\lambda](1,s)\\
&=G_{D}[\lambda](t,s)+\frac{\partial}{\partial s} G_{D}[\lambda](t,1)\, G_{M_{2}}[\lambda](1,s),\quad \forall (t,s)\in I\times I,
\end{aligned}
\end{equation}
\end{theorem}
\begin{corollary}\label{corolario21}
	The following inequality holds 
	\begin{equation}\label{e-des-GM2-GD}
	G_{M_{2}}[\lambda](t,s)<G_{D}[\lambda](t,s)<0,\;\; \forall (t,s)\in (0,1)\times(0,1),\;\;	\forall  \lambda<\lambda_{0}^{M_{2}}.
	\end{equation}
\end{corollary}
\begin{remark}
The above inequality between $G_{M_{2}}[\lambda]$ and $G_{D}[\lambda]$ can be deduced from \cite[Corollaries 4.7, 4.8, 4.13]{kkkk}. Moreover the expression \eqref{e-M2-D} relating $G_{M_{2}}[\lambda]$ and $G_{D}[\lambda]$ is different from the one obtained in that reference. 

However, as far as we know, there is no expression in the literature that relate $G_{M_{1}}$ and $G_{D}$ and, as a consequence, equality  \eqref{7} and  inequality \eqref{e-des-GM1-GD} are new.
\end{remark}
Analogously to previous sections, we can relate expressions of the Green's function of the Dirichlet problem  and the  ones of the Mixed problems. 
\begin{theorem}
If operator $L[\lambda]$ is nonresonant in $X_D$ and $X_{M_2}$, then 
 \begin{equation*}
 \begin{aligned}
	G_{D}[\lambda](t,s)=&G_{M_{2}}[\lambda](t,s)+r_{8}[\lambda](t)\, \frac{\partial }{\partial t}G_{D}[\lambda](1,s)\\
	=&G_{M_{2}}[\lambda](t,s)-G_{M_{2}}[\lambda](t,1)\, \frac{\partial }{\partial t}G_{D}[\lambda](1,s),\;\; t,s\in I.
\end{aligned}
\end{equation*}
\end{theorem}

\begin{theorem}
	If operator $L[\lambda]$ is nonresonant in $X_D$ and $X_{M_1}$, then
\begin{equation*}
\begin{aligned}
	G_{D}[\lambda](t,s)=&G_{M_{1}}[\lambda](t,s)+r_{9}[\lambda](t)\, \frac{\partial}{\partial t} G_{D}[\lambda](0,s)\\
	=&G_{M_{1}}[\lambda](t,s)+G_{M_{1}}[\lambda](t,0)\, \frac{\partial}{\partial t} G_{D}[\lambda](0,s),\;\; t,s \in I.
	\end{aligned}
\end{equation*}	
\end{theorem} 
\begin{remark}
	Notice that from two previous results we can deduce Corollaries \ref{corolario3} and \ref{corolario21}.
\end{remark}


\subsection{Neumann and Mixed problems}
In this section, arguing in a similar manner than in previous ones, we can relate the expression of the Green's functions of the Neumann problem and the ones of the corresponding Mixed ones.
\begin{theorem}
 Assume that operator $L[\lambda]$ is nonresonant in $X_{N}$ and $X_{M_{1}}$. Then it holds that
 \begin{equation}\label{8}
 \begin{aligned}
 G_{M_{2}}[\lambda](t,s)=&G_{N}[\lambda](t,s)+r_{5}[\lambda](t)\, \frac{\partial }{\partial t}G_{M_{2}}[\lambda](0,s)\\
 =&G_{N}[\lambda](t,s)+G_{N}[\lambda](t,0)\, \frac{\partial }{\partial t}G_{M_{2}}[\lambda](0,s),\;\; \forall (t,s)\in I\times I.
 \end{aligned}
 \end{equation}	
\end{theorem}

\begin{corollary}\label{corolario1}
 The following inequality holds 	
\begin{equation}\label{e-N-M2}
	G_{N}[\lambda](t,s)<G_{M_{2}}[\lambda](t,s)<0,\;\; \forall (t,s)\in (0,1]\times (0,1],\;\; \forall \lambda<\lambda_{0}^{N}.
\end{equation} 
\end{corollary}
\begin{proof}
	We know that, $G_{M_{2}}[\lambda](t,s)<0$ for all $(t,s)\in (0,1]\times (0,1]$ if and only if $\lambda<\lambda_{0}^{M_{2}}$ (\cite[Corollary 4.6]{kkkk}). Since $G_{M_{2}}[\lambda](0,s)=0$ we deduce that $\dfrac{\partial }{\partial t}G_{M_{2}}[\lambda](0,s)<0$ for such $\lambda$. In addition,  $\lambda_{0}^{N}<\lambda_{0}^{M_{1}}$ (\cite[Remark 4.19]{kkkk}). Therefore, using equality \eqref{8} we obtain the result.	
\end{proof}
Analogously, for Mixed 1 problem we have the following results. 
\begin{theorem}
Assume that $L[\lambda]$ is nonresonant in $X_{M_{1}}$ and $X_{N}$, then
\begin{equation}\label{rel-GM1-GN}
\begin{aligned}
G_{M_{1}}[\lambda](t,s)=&G_{N}[\lambda](t,s)+r_{6}[\lambda](t)\,\frac{\partial }{\partial t} G_{M_{1}}[\lambda](1,s)\\
=&G_{N}[\lambda](t,s)-G_{N}[\lambda](t,1)\,\frac{\partial }{\partial t} G_{M_{1}}[\lambda](1,s),\quad \forall (t,s)\in I\times I.
\end{aligned}
\end{equation}
\end{theorem} 
\begin{corollary}\label{corolario2}
The following equality is fulfilled 
\begin{equation}\label{e-des-GN-GM1}
	G_{N}[\lambda](t,s)<G_{M_{1}}[\lambda](t,s)<0,\;\; \forall (t,s)\in [0,1)\times [0,1),\;\; \forall \lambda<\lambda_{0}^{N}. 
\end{equation}
\end{corollary}
\begin{remark}
Inequality \eqref{e-des-GN-GM1} can be deduced from \cite[Corollaries 4.5, 4.8, 4.13]{kkkk}. Identities \eqref{8} and \eqref{rel-GM1-GN} together with inequality \eqref{e-N-M2}  are new.
\end{remark}
By the reciprocal process, we can obtain additional relations between  the Green's function of the Neumann problem and the ones of the Mixed problems as follows. 
\begin{theorem}\label{th1}
If operator $L[\lambda]$ is nonresonant in $X_{N}$ and $X_{M_{2}}$, then 
\begin{equation*}
\begin{aligned}
G_{N}[\lambda](t,s)&=G_{M_{2}}[\lambda](t,s)+r_{7}[\lambda](t)\, G_{N}[\lambda](0,s)\\
&=G_{M_{2}}[\lambda](t,s)-\frac{\partial }{\partial s}G_{M_{2}}[\lambda](t,0)\, G_{N}[\lambda](0,s),\;\; t,s\in I.
\end{aligned}
\end{equation*}
\end{theorem}
\begin{theorem}\label{th2}
If operator $L[\lambda]$ is nonresonant in $X_{N}$ and $X_{M_{1}}$, then	
\begin{equation*}
\begin{aligned}
G_{N}[\lambda](t,s)&=G_{M_{1}}[\lambda](t,s)+r_{10}[\lambda](t)\, G_{N}[\lambda](1,s)\\
&=G_{M_{1}}[\lambda](t,s)-\frac{\partial }{\partial s}G_{M_{1}}[\lambda](t,1)\, G_{N}[\lambda](1,s),\;\; t,s \in I.
\end{aligned}
\end{equation*}
\end{theorem}
\begin{remark}
	Notice that Corollaries \ref{corolario1} and \ref{corolario2}  can be deduced from Theorem \ref{th1} and \ref{th2} (respectively).
\end{remark}
\subsection{Periodic and Neumann problems}
Concerning the Neumann and Periodic problems and  arguing as before, we arrive to the next theorem. 
\begin{theorem}
If operator $L[\lambda]$ is nonresonant both in $X_{N}$ and $X_{P}$, the following equality is fulfilled
\begin{equation}\label{gp-gd}
\begin{aligned}
G_{P}[\lambda](t,s)=&G_{N}[\lambda](t,s)+\left(r_{5}[\lambda](t)+r_{6}[\lambda](t)\right)\,  \frac{\partial}{\partial t} G_{P}[\lambda](0,s)\\
=&G_{N}[\lambda](t,s)+\left(G_{N}[\lambda](t,0)-G_{N}[\lambda](t,1)\right)\,  \frac{\partial}{\partial t} G_{P}[\lambda](0,s),\;\; \forall (t,s)\in  I\times I.
\end{aligned}
\end{equation}
\end{theorem}
\begin{remark}
From \eqref{gp-gd} and due to the symmetry of $G_{P}[\lambda]$ and $G_{N}[\lambda]$, we deduce that
\begin{equation*}
\small
\left(G_{N}[\lambda](t,0)-G_{N}[\lambda](t,1) \right)\, \frac{\partial}{\partial t} G_{P}[\lambda](0,s)=\left(G_{N}[\lambda](s,0)-G_{N}[\lambda](s,1) \right)\, \frac{\partial}{\partial t} G_{P}[\lambda](0,t),\;\; \forall (t,s)\in I\times I.
\end{equation*}
If $\frac{\partial}{\partial t} G_{P}[\lambda](0,t)\neq 0$ and  $\frac{\partial}{\partial t} G_{P}[\lambda](0,s)\neq 0$, then 
$$\dfrac{G_{N}[\lambda](t,0)-G_{N}[\lambda](t,1)}{\frac{\partial }{\partial t}G_{P}[\lambda](0,t)}=\dfrac{G_{N}[\lambda](s,0)-G_{N}[\lambda](s,1)}{\frac{\partial }{\partial t}G_{P}[\lambda](0,s)}=c_{1}\in \mathbb{R}.$$
We know that $\frac{\partial}{\partial t} G_{P}[\lambda](t,s)=\frac{\partial}{\partial s} G_{P}[\lambda](s,t)$ and $\frac{\partial}{\partial s} G_{P}[\lambda](t,s)=\frac{\partial}{\partial t} G_{P}[\lambda](s,t)$. Then 
\begin{equation*}
G_{N}(t,0)-G_{N}(t,1)=c_{1}\, \frac{\partial }{\partial t} G_{P}[\lambda](0,t)=c_{1}\, \frac{\partial }{\partial s} G_{P}[\lambda](t,0).
\end{equation*}
\end{remark}


With the reverse process we arrive at the following result.
\begin{theorem}
Assume that $L[\lambda]$ is nonresonant in $X_{N}$ and $X_{P}$, then
\begin{equation*}
\begin{aligned}
G_{N}[\lambda](t,s)&=G_{P}[\lambda](t,s)+r_{3}[\lambda](t)\, (G_{N}[\lambda](0,s)-G_{N}[\lambda](1,s))\\
&=G_{P}[\lambda](t,s)-\frac{\partial }{\partial s}G_{P}[\lambda](t,0)\, (G_{N}[\lambda](0,s)-G_{N}[\lambda](1,s)),\;\; \forall (t,s) \in I\times I.
\end{aligned} 
\end{equation*}
\end{theorem}
\subsection{Periodic and Mixed problems}	
The same arguments of the previous subsections are applicable to the Periodic and Mixed 1 problems. We omit the proof, which is analogous to those of previous cases.
\begin{theorem}
Assume that $L[\lambda]$ is nonresonant in $X_{P}$ and $X_{M_{1}}$. Then it holds that
\begin{equation*}
\begin{aligned}
	G_{M_{1}}[\lambda](t,s)=&G_{P}[\lambda](t,s)+r_{3}[\lambda](t)\, G_{M_{1}}[\lambda](0,s)-r_{4}[\lambda](t)\, \frac{\partial }{\partial t}G_{M_{1}}[\lambda](1,s)\\
	=&G_{P}[\lambda](t,s)-\frac{\partial }{\partial s}G_{P}[\lambda](t,0)\, G_{M_{1}}[\lambda](0,s)-G_{P}[\lambda](t,0)\, \frac{\partial }{\partial t}G_{M_{1}}[\lambda](1,s),\;\; \forall (t,s)\in I\times I.
\end{aligned}
\end{equation*}
\end{theorem}

Next example shows that, in general, the Green's functions of Periodic and Mixed 1 problems are not comparable.
\begin{example}\label{11}
	We consider the differential equation $u''(t)-m^{2}\, u(t)=0$, $t\in I$ and $m\in (0,\infty)$. In this case, $a(t)=-m^{2}$, $t\in I$, $\lambda=0$ and $m\in(0,\infty)$.
	
	\begin{figure}
	\subfloat{
		\includegraphics[width=8cm]{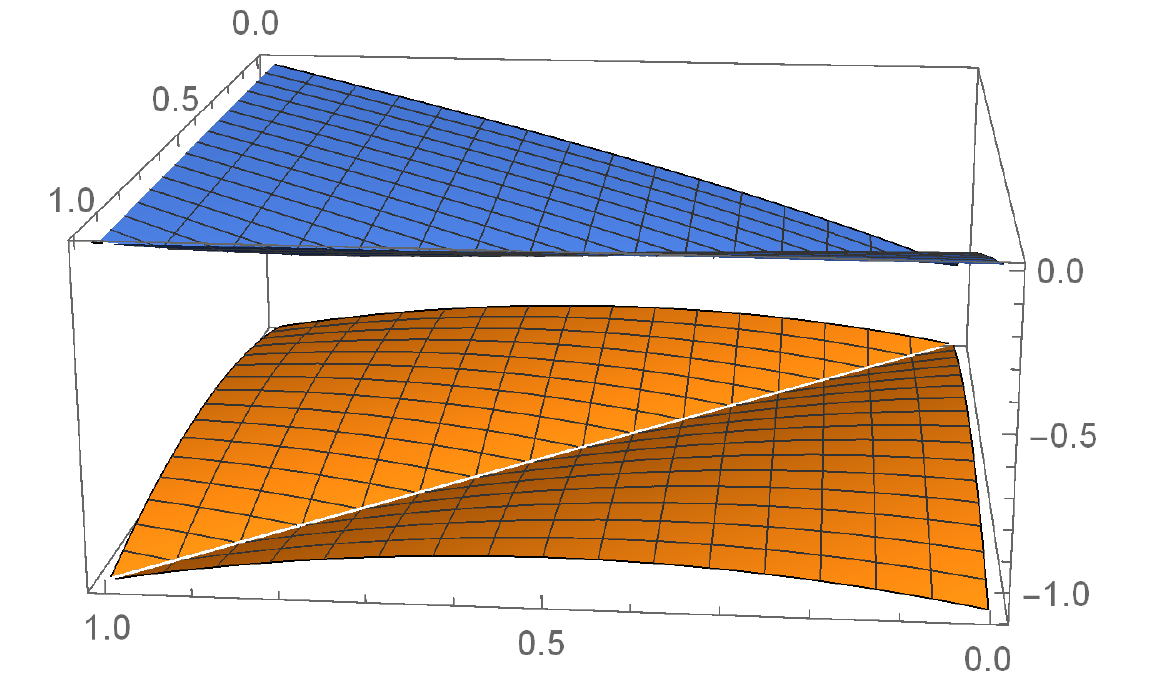}}	
\subfloat{	\includegraphics[width=7.5cm]{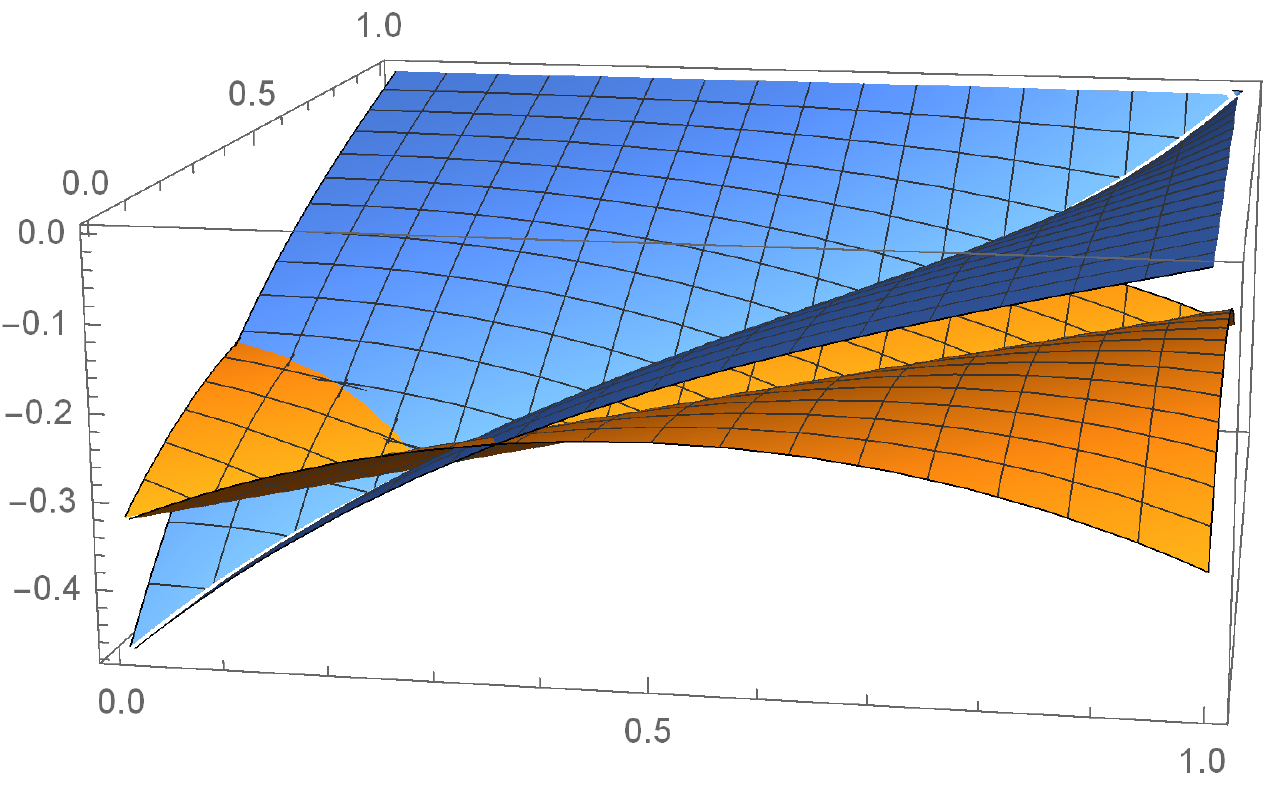}}
		\caption{The blue graph corresponds to function $G_{M_{1}}$ and the orange graph represents the  function  $G_{P}$ on $I\times I$. The figure on the left is the case $m=1$ and the figure on the right is the case $m=2$.}
		\label{21}
	\end{figure}
 
 Green's functions $G_{P}$ and $G_{M_{1}}$ are comparable for small values of $m$. Figure \ref{21} represents the Green's functions $G_{P}$ and $G_{M_{1}}$ for $m=1$ (in which case $G_{P}<G_{M_{1}}$) and for $m=2$ (which are not comparable).

\begin{figure}
\subfloat{\includegraphics[width=9cm]{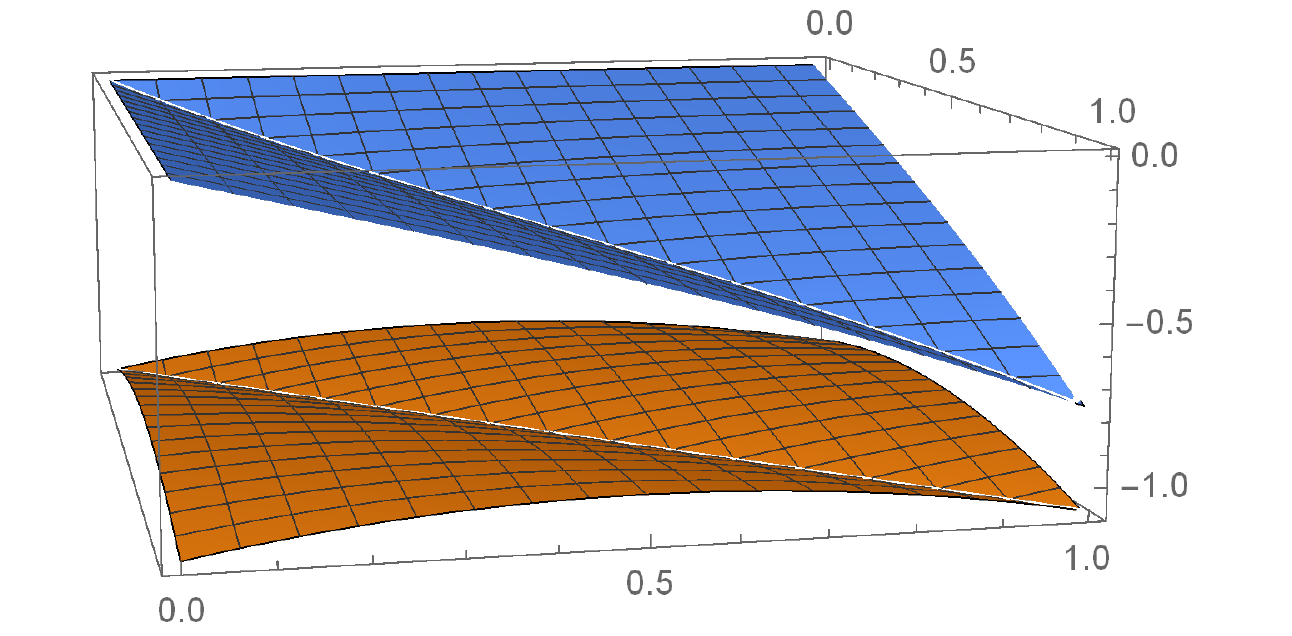}}
\subfloat{\includegraphics[width=7cm]{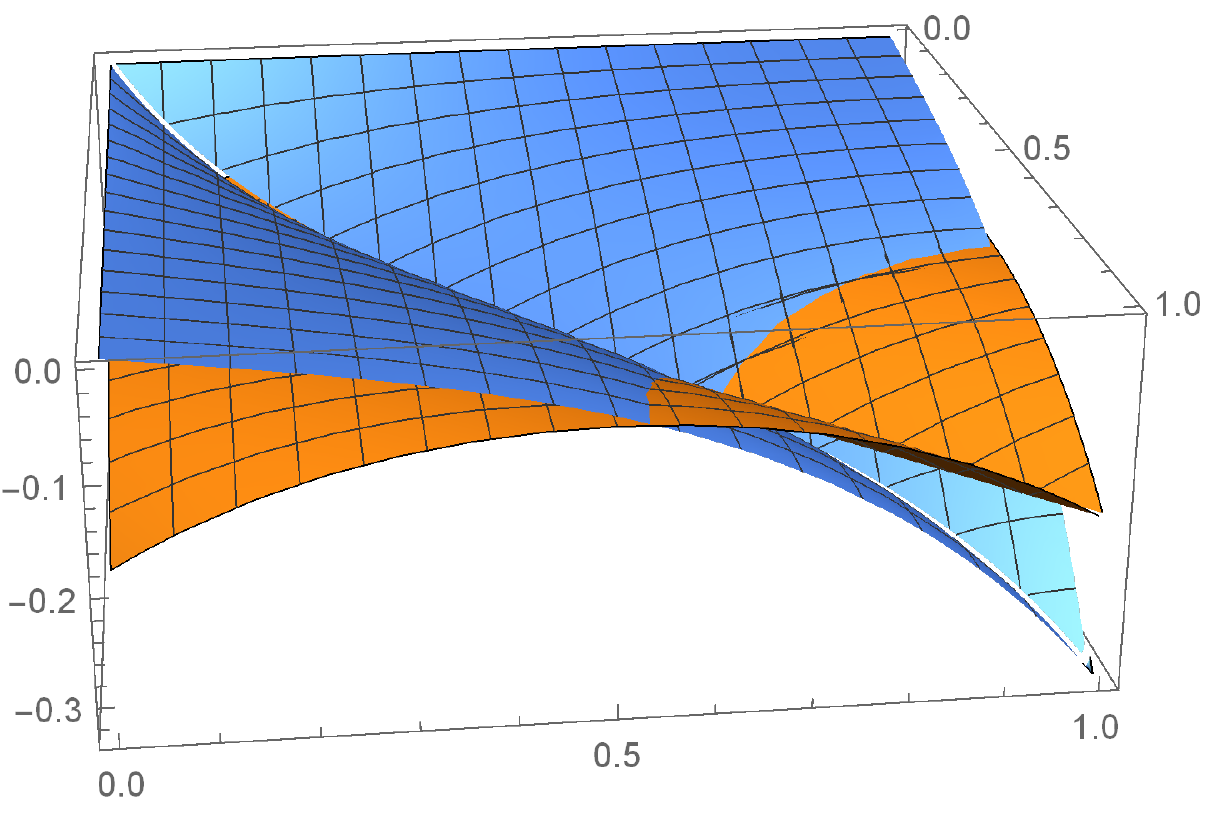}}
\caption{The blue graph corresponds to the function $G_{M_{2}}$ and the orange graph represents the  function  $G_{P}$ on $I\times I$. The figure on the left is the case $m=1$ and the figure on the right is the case $m=3$.}
	\label{23}
\end{figure}

\end{example}
Analogously, we study the relationship between Green's functions of Periodic and Mixed 2. 

\begin{theorem}
Assume that $L[\lambda]$ is nonresonant in $X_{P}$ and $X_{M_{2}}$, then it holds that
\begin{equation*}
\begin{aligned}
G_{M_{2}}[\lambda](t,s)&=G_{P}[\lambda](t,s)-r_{3}[\lambda](t)\, G_{M_{2}}[\lambda](1,s)+r_{4}[\lambda](t)\, \frac{\partial}{\partial t}G_{M_{2}}[\lambda](0,s)\\
&=G_{P}[\lambda](t,s)+\frac{\partial}{\partial s} G_{P}[\lambda](t,0)\, G_{M_{2}}[\lambda](1,s)+G_{P}[\lambda](t,0)\, \frac{\partial }{\partial t}G_{M_{2}}[\lambda](0,s),\;\; \forall (t,s)\in I\times I.
\end{aligned}
\end{equation*}
\end{theorem}
The above equation is analogous to the equation relating Periodic to Mixed 1. So, in general, Green's functions of the Periodic and Mixed 2 problems will not be comparable either.
\begin{example}
We use in this example the same equation as in  Example \ref{11}. 
 Green's functions $G_{P}$ and $G_{M_{2}}$ are comparable for small values of $m$. Figure \ref{23} represents the Green's functions $G_{P}$ and $G_{M_{2}}$ for $m=1$ (in which case $G_{P}<G_{M_{2}}$) and for $m=3$ (which are not comparable).
\end{example}
Finally for the reverse process, we can obtain additional relations for the  Green's function of the Periodic  and the ones related to Mixed problems.
\begin{theorem}
If operator $L[\lambda]$ is nonresonant both in $X_{P}$ and $X_{M_{2}}$, then
\begin{equation*}
\begin{aligned}
G_{P}[\lambda](t,s)=&G_{M_{2}}[\lambda](t,s)+r_{7}[\lambda](t)\, G_{P}[\lambda](1,s)+r_{8}[\lambda](t)\, \frac{\partial }{\partial t}G_{P}[\lambda](0,s)\\
=&G_{M_{2}}[\lambda](t,s)-\frac{\partial }{\partial s}G_{M_{2}}[\lambda](t,0)\, G_{P}[\lambda](1,s)-G_{M_{2}}[\lambda](t,1)\, \frac{\partial }{\partial t}G_{P}[\lambda](0,s),\;\; t,s\in I.
\end{aligned}
\end{equation*}
\end{theorem}
\begin{theorem}
If operator $L[\lambda]$ is nonresonant both in $X_{P}$ and $X_{M_{1}}$, then	
\begin{equation*}
\begin{aligned}
G_{P}[\lambda](t,s)=&G_{M_{1}}[\lambda](t,s)+r_{9}[\lambda](t)\, \frac{\partial}{\partial t}G_{P}[\lambda](1,s)+r_{10}[\lambda](t)\, G_{P}[\lambda](0,s)\\
=&G_{M_{1}}[\lambda](t,s)+G_{M_{1}}[\lambda](t,0)\, \frac{\partial}{\partial t}G_{P}[\lambda](1,s)-\frac{\partial}{\partial s}G_{M_{1}}[\lambda](t,1)\, G_{P}[\lambda](0,s),\;\; (t,s) \in I\times I.
\end{aligned}
\end{equation*}
\end{theorem}
\section{Alternative decomposition of Green's functions}
This section is devoted to deduce additional relationships between the expressions of the Green's functions related to the different boundary value  conditions studied in previous section. The main difference consists on the fact that in this case, instead of Lemma \ref{e-sol-hi} as in previous section, we will use Theorem \ref{th_Ci}. 

It is important to point out that in this situation, as an application of equality  \eqref{e-formula} we will be able to express any considered Green's function in an explicit way from any other one.

The obtained expressions will be different to the ones deduce in previous section.
\subsection{Dirichlet and Mixed problems}
We start this subsection by expressing the Green's function of Mixed 2 problem in terms of the Green's function of Dirichlet problem.
\begin{theorem}
If operator $L[\lambda]$ is nonresonant in $X_{D}$ and $r'_{2}[\lambda](1)\neq 0$, then the following equality holds
\begin{equation}\label{e-formula2}
\begin{aligned}
G_{M_{2}}[\lambda](t,s)=&G_{D}[\lambda](t,s)-\frac{r_{2}[\lambda](t)}{r'_{2}[\lambda](1)}\, \frac{\partial}{\partial t}G_{D}[\lambda](1,s)\\
=&G_{D}[\lambda](t,s)-\dfrac{\frac{\partial }{\partial s}G_{D}\lambda](t,1)}{\frac{\partial^{2}}{\partial s\partial t}G_{D}[\lambda](1,1)}\, \frac{\partial}{\partial t}G_{D}[\lambda](1,s),\;\; \forall (t,s)\in I\times I. 
\end{aligned}
\end{equation}
\end{theorem}
\begin{proof}
We write the Mixed 2 problem based on the Dirichlet problem as follows
\begin{equation}\label{problema1}
L[\lambda]\, u(t)=\sigma(t),\;\; \text{a.e.} \;\; t\in I,\;\; u(0)=0,\;\; u(1)=u(1)+u'(1).
\end{equation}
Using the notation of Theorem \ref{th_Ci}, we have that in this case $C_{1}(u)=0$, $C_{2}(u)=u(1)+u'(1)$ and $\delta_{1}=\delta_{2}=1$. Moreover, $\omega_{1}(t)=r_{1}[\lambda](t)$, $\omega_{2}(t)=r_{2}[\lambda](t)$ and the matrix $A_{D}^{1}$ in this case is
\begin{equation*}
A_{D}^{1}=\begin{pmatrix}
0 && 0\\
r'_{1}[\lambda](1) && 1+r'_{2}[\lambda](1)
\end{pmatrix}
\end{equation*}
and $|I-A_{D}^{1}[\lambda]|=-r'_{2}[\lambda](1)\neq 0$. So, 
\begin{equation*}
b_{D}^{1}=(I-A_{D}^{1}[\lambda])^{-1}=\begin{pmatrix}
1 && 0\\
-\frac{r'_{1}[\lambda](1)}{r'_{2}[\lambda](1)} && -\frac{1}{r'_{2}[\lambda](1)}
\end{pmatrix}.
\end{equation*}
In consequence, as a direct application of the equality \eqref{e-formula} we obtain the result.
\end{proof}
\begin{corollary}
For all $\lambda<\lambda_{0}^{M_{2}}$, we infer that $r'_{2}[\lambda](1)>0$.	
\end{corollary}
\begin{proof}
From Corollary \ref{corolario21} we have that $G_{M_{2}}[\lambda]<G_{D}[\lambda]<0$ for all $\lambda<\lambda_{0}^{M_{2}}$  and, as a direct consequence, $\frac{\partial }{\partial t} G_{D}[\lambda](1,s)>0$. Lemma \ref{signo1} says us that $r_{2}[\lambda]>0$ on $(0,1]$ for all $\lambda<\lambda_{0}^{D}$. Since (\cite[page 108]{A}) $\lambda_{0}^{M_{2}}<\lambda_{0}^{D}$ we deduce, from equality \eqref{e-formula2} that $r'_{2}[\lambda](1)>0$.
\end{proof}

Similarly, we study the Mixed 1 problem  as a function of the Dirichlet one.
\begin{theorem}
	If operator $L[\lambda]$ is nonresonant in $X_{D}$ and $r'_{1}[\lambda](0)\neq 0$, then the following equality holds
\begin{equation*}
\begin{aligned}
G_{M_{1}}[\lambda](t,s)=&G_{D}[\lambda](t,s)-\frac{r_{1}[\lambda](t)}{r'_{1}[\lambda](0)}\, \frac{\partial }{\partial t}G_{D}[\lambda](0,s)\\
=&G_{D}[\lambda](t,s)-\dfrac{\frac{\partial }{\partial s}G_{D}\lambda](t,0)}{\frac{\partial^{2}}{\partial s\partial t}G_{D}[\lambda](0,0)}\, \frac{\partial }{\partial t}G_{D}[\lambda](0,s),\;\; \forall (t,s)\in I\times I.
\end{aligned}
\end{equation*}
\end{theorem}
\begin{proof}
 Let us rewrite Mixed 1 problem in the following way 
\begin{equation}\label{problema2}
L[\lambda]\, u(t)=\sigma(t),\;\; \text{a.e.} \;\; t\in I,\;\; u(0)=u(0)+u'(0),\;\; u(1)=0.
\end{equation}
In this case, we have that $C_{1}(u)=u(0)+u'(0)$, $C_{2}(u)=0$ and $\delta_{1}=\delta_{2}=1$. Moreover, $\omega_{1}(t)=r_{1}[\lambda](t)$, $\omega_{2}(t)=r_{2}[\lambda](t)$ and the matrix $A_{D}^{2}[\lambda]$ is
\begin{equation*}
A_{D}^{2}[\lambda]=\begin{pmatrix}
1+r'_{1}[\lambda](0) && r'_{2}[\lambda](0)\\
0 && 0
\end{pmatrix}
\end{equation*}
and $|I-A_{D}^{2}[\lambda]|=-r'_{1}[\lambda](0)\neq 0$. So, 
\begin{equation*}
b_{D}^{2}=(I-A_{D}^{2}[\lambda])^{-1}=\begin{pmatrix}
-\frac{1}{r'_{1}[\lambda](0)} && -\frac{r'_{2}[\lambda](0)}{r'_{1}[\lambda](0)}\\
0 && 1
\end{pmatrix}.
\end{equation*}
Therefore, using \eqref{e-formula}, we deduce the equality.
\end{proof} 
\begin{corollary}
	For all $\lambda<\lambda_{0}^{M_{1}}$, we infer that $r'_{1}[\lambda](0)<0$.	
\end{corollary}
\begin{proof}
From Lemma \ref{signo2} we know that $r_{1}[\lambda]>0$ on $[0,1)$ for all $\lambda<\lambda_{0}^{D}$. Corollary \ref{corolario3} ensures that $G_{M_{1}}[\lambda]<G_{D}[\lambda]$ for all $\lambda<\lambda_{0}^{M_{1}}$. Since $\lambda_{0}^{M_{1}}<\lambda_{0}^{D}$ (\cite[page 108]{A}) we arrive at the result. 
\end{proof}
\begin{remark}
In problem \eqref{problema1} we can do the calculations in a simpler way by taking $C_{1}(u)=u(0)+u'(0)$, $C_{2}(u)=0$, $\delta_{1}=1$ and $\delta_{2}=0$. The same can be done with problem \eqref{problema2} by taking $C_{1}(u)=C_{2}(u)=u(1)+u'(1)$, $\delta_{1}=0$ and $\delta_{2}=1$.
\end{remark}
We now do the process backwards by writing the Dirichlet problem based on the Mixed ones.
We arrive at the following results.
\begin{theorem}
If operator $L[\lambda]$ is nonresonant in $X_{M_{2}}$ and $r_{8}[\lambda](1)\neq 0$, then
\begin{equation*}
\begin{aligned}
G_{D}[\lambda](t,s)=&G_{M_{2}}[\lambda](t,s)-\frac{r_{8}[\lambda](t)}{r_{8}[\lambda](1)}\, G_{M_{2}}[\lambda](1,s)\\
=&G_{M_{2}}[\lambda](t,s)-\frac{G_{M_{2}}[\lambda](t,1)}{G_{M_{2}}[\lambda](1,1)}\, G_{M_{2}}[\lambda](1,s),\;\; \forall (t,s)\in I\times I.
\end{aligned}
\end{equation*}
\end{theorem}

\begin{theorem}
	If operator $L[\lambda]$ is nonresonant in $X_{M_{1}}$ and $r_{9}[\lambda](0)\neq 0$, then
\begin{equation*}
\begin{aligned}
G_{D}[\lambda](t,s)=&G_{M_{1}}[\lambda](t,s)-\frac{r_{9}[\lambda](t)}{r_{9}[\lambda](0)}\, G_{M_{1}}[\lambda](0,s)\\
=&G_{M_{1}}[\lambda](t,s)-\dfrac{G_{M_{1}}[\lambda](t,0)}{G_{M_{1}}[\lambda](0,0)}\, G_{M_{1}}[\lambda](0,s),\;\; \forall (t,s)\in I\times I.
\end{aligned}
\end{equation*}
\end{theorem}
\subsection{Neumann and Dirichlet problems}
In this case we study the relationships between the Green's function of the Neumann and Dirichlet problems. Reasoning as in the previous subsection we have the next result. 
\begin{theorem}
	If operator $L[\lambda]$ is nonresonant in $X_{D}$ and $$|I-A_{D}^{3}[\lambda]|:=r'_{1}[\lambda](0)\, r'_{2}[\lambda](1)-r'_{2}[\lambda](0)\, r'_{1}[\lambda](1)\neq 0,$$ then it holds that 
\begin{equation*}
\small
\begin{aligned}
G_{N}[\lambda](t,s)=&G_{D}[\lambda](t,s)-\frac{r'_{2}[\lambda](1)}{|I-A_{D}^{3}[\lambda]|}\, r_{1}[\lambda](t)\, \frac{\partial}{\partial t} G_{D}[\lambda](0,s)+\frac{r'_{1}[\lambda](1)}{|I-A_{D}^{3}[\lambda]|}\, r_{2}[\lambda](t)\, \frac{\partial}{\partial t} G_{D}[\lambda](0,s)\\
&+\frac{r'_{2}[\lambda](0)}{|I-A_{D}^{3}[\lambda]|}\, r_{1}[\lambda](t)\, \frac{\partial }{\partial t} G_{D}[\lambda](1,s)-\frac{r'_{1}[\lambda](0)}{|I-A_{D}^{3}[\lambda]|}\, r_{2}[\lambda](t)\, \frac{\partial}{\partial t} G_{D}[\lambda](1,s)\\
=&G_{D}[\lambda](t,s)+\frac{1}{|I-A_{D}^{3}[\lambda]|}\,\frac{\partial^{2} }{\partial s\partial t} G_{D}[\lambda](1,1)\,  \frac{\partial }{\partial s} G_{D}[\lambda](t,0)\, \frac{\partial}{\partial t} G_{D}[\lambda](0,s)\\
&-\frac{1}{|I-A_{D}^{3}[\lambda]|}\,\frac{\partial^{2} }{\partial s\partial t} G_{D}[\lambda](1,0)\, \frac{\partial }{\partial s} G_{D}[\lambda](t,1)\, \frac{\partial}{\partial t} G_{D}[\lambda](0,s)\\
&-\frac{1}{|I-A_{D}^{3}[\lambda]|}\,\frac{\partial^{2} }{\partial s\partial t} G_{D}[\lambda](0,1)\,  \frac{\partial }{\partial s} G_{D}[\lambda](t,0)  \, \frac{\partial }{\partial t} G_{D}[\lambda](1,s)\\
&+\frac{1}{|I-A_{D}^{3}[\lambda]|}\,\frac{\partial^{2} }{\partial s\partial t} G_{D}[\lambda](0,0)  \, \frac{\partial }{\partial s}G_{D}[\lambda](t,1)\, \frac{\partial}{\partial t} G_{D}[\lambda](1,s),\;\; \forall (t,s)\in I\times I.
\end{aligned}
\end{equation*}	
\end{theorem}
We now reverse the process by studying the Dirichlet problem as a function of the Neumann one 
and applying  analogous calculations.  
\begin{theorem}
Assume that $L[\lambda]$ is nonresonant in $X_{N}$ and  $$|I-A_{N}^{1}[\lambda]|:=r_{5}[\lambda](0)\, r_{6}[\lambda](1)-r_{6}[\lambda](0)\, r_{5}[\lambda](1)\neq 0,$$ then
\begin{equation*}
\begin{aligned}
G_{D}[\lambda](t,s)=&G_{N}[\lambda](t,s)-\frac{1}{|I-A_{N}^{1}[\lambda]|}\Big(\, r_{5}[\lambda](t)\,\Big(-r_{6}[\lambda](1)\, G_{N}[\lambda](0,s)+r_{6}[\lambda](0)\, G_{N}[\lambda](1,s)\Big) \Big.\\
&\Big.+r_{6}[\lambda](t)\,\Big(r_{5}[\lambda](1)\,G_{N}[\lambda](0,s)-r_{5}[\lambda](0)\,G_{N}[\lambda](1,s) \Big)\Big)\\
=&G_{N}[\lambda](t,s)-\frac{1}{|I-A_{N}^{1}[\lambda]|}\Big( G_{N}[\lambda](t,0)\,\Big(G_{N}[\lambda](1,1)\, G_{N}[\lambda](0,s)-G_{N}[\lambda](0,1)\, G_{N}[\lambda](1,s)\Big) \Big.\\
&\Big.-G_{N}[\lambda](t,1)\,\Big(G_{N}[\lambda](1,0)\,G_{N}[\lambda](0,s)-G_{N}[\lambda](0,0)\,G_{N}[\lambda](1,s) \Big)\Big),\;\; \forall (t,s) \in I\times I.
\end{aligned}
\end{equation*}
\end{theorem}
\subsection{Periodic and Dirichlet problems}
In this section we give a relationship between $G_{P}[\lambda]$ and $G_{D}[\lambda]$ following the same steps than in previous sections. 

\begin{theorem}
	Assume that $L[\lambda]$ is nonresonant in $X_{D}$ and $$|I-A_{D}^{4}[\lambda]|:=2r'_{1}[\lambda](1)+r'_{2}[\lambda](1)-r'_{1}[\lambda](0)\neq 0,$$ then 
\begin{equation*}
\small
\begin{aligned}
G_{P}[\lambda](t,s)=&G_{D}[\lambda](t,s)+\frac{(r_{1}[\lambda](t)+r_{2}[\lambda](t))}{|I-A_{D}^{4}[\lambda]|}\, \left(\frac{\partial}{\partial t} G_{D}[\lambda](0,s)-\frac{\partial }{\partial t} G_{D}[\lambda](1,s)\right)\\
=&G_{D}[\lambda](t,s)+\dfrac{ \left(\frac{\partial}{\partial s} G_{D}[\lambda](t,1)-\frac{\partial}{\partial s} G_{D}[\lambda](t,0)\right)}{|I-A_{D}^{4}[\lambda]|}\,\left(\frac{\partial}{\partial t} G_{D}[\lambda](0,s)-\frac{\partial }{\partial t} G_{D}[\lambda](1,s)\right),\;\; \forall (t,s)\in I\times I.
\end{aligned}
\end{equation*}
\end{theorem}
\begin{remark}
Notice that from \eqref{e-Gp-Gd}, using the last equality we have that 
\begin{equation*}
G_{P}[\lambda](1,s)=\frac{1}{|I-A_{D}^{4}[\lambda]|}\,\left[\frac{\partial}{\partial t} G_{D}[\lambda](0,s)- \frac{\partial}{\partial t} G_{D}[\lambda](1,s)\right].
\end{equation*}
\end{remark}
Finally, doing the process backwards by studying the Dirichlet problem as a function of the Periodic one 
we obtain the next theorem. 
\begin{theorem}\label{teorema2}
	If operator $L[\lambda]$ is nonresonant in $X_{P}$ and $r_{4}[\lambda](1)\neq 0$, then 
\begin{equation*}
\begin{aligned}
G_{D}[\lambda](t,s)=&G_{P}[\lambda](t,s)-\frac{r_{4}[\lambda](t)}{r_{4}[\lambda](1)}\, G_{P}[\lambda](1,s)\\
=&G_{P}[\lambda](t,s)-\frac{G_{P}[\lambda](t,0)}{G_{P}[\lambda](1,0)}\, G_{P}[\lambda](1,s),\;\; \forall (t,s)\in I\times I.
\end{aligned}
\end{equation*}
\end{theorem}
\begin{remark}
From Theorem \ref{teorema2} we deduce Theorem \ref{teorema1}:
\begin{equation*}
G_{P}[\lambda](t,s)<G_{D}[\lambda](t,s)<0,\;\; \forall (t,s) \in  (0,1)\times (0,1),\;\;\forall \lambda<\lambda_{0}^{P}.
\end{equation*}
\end{remark}
\subsection{Neumann and Mixed problems}
We do the same as before to study the relationship between the Green's functions of Neumann and Mixed problems 1 and 2.
\begin{theorem}
Assume that $L[\lambda]$ is nonresonant $X_{N}$ and $r_{5}[\lambda](0)\neq 0$, then it holds that
\begin{equation*}
\begin{aligned}
G_{M_{2}}[\lambda](t,s)=&G_{N}[\lambda](t,s)-\frac{r_{5}[\lambda](t)}{r_{5}[\lambda](0)}\, G_{N}[\lambda](0,s)\\
=&G_{N}[\lambda](t,s)-\frac{G_{N}[\lambda](t,0)}{G_{N}[\lambda](0,0)}\, G_{N}[\lambda](0,s),\;\; \forall (t,s)\in I\times I.
\end{aligned}
\end{equation*}
\end{theorem}
Using previous expression, we have another proof of Corollary \ref{corolario1}. Indeed, we know that $G_{N}[\lambda]<0$ for all $\lambda<\lambda_{0}^{N}$ and $r_{5}[\lambda](t)=G_{N}[\lambda](t,0)<0$, using the above equality we deduce for all $\lambda<\lambda_{0}^{N}$ that
\begin{center}
$G_{N}[\lambda](t,s)<G_{M_{2}}[\lambda](t,s),\;\;$ for all $(t,s)\in I\times I$.
\end{center}

\begin{remark}
	We will give as a consequence of the last equality a proof of Corollary \ref{corolario2}. Taking into account that $G_{N}[\lambda]<0$ for all $\lambda<\lambda_{0}^{N}$ and $r_{6}[\lambda](t)>0$, $t\in I$, it follows that for all $\lambda<\lambda_{0}^{N}$
	\begin{center}
		$G_{N}[\lambda](t,s)<G_{M_{1}}[\lambda](t,s)<0,\;\;$ for all $(t,s)\in [0,1)\times [0,1)$. 
	\end{center}
\end{remark}

Doing the calculations analogously for the Mixed 1 problem as a function of Neumann problem
we have the relationship between Green's functions given in the next theorem.
\begin{theorem}
Assume that $L[\lambda]$ is nonresonant in $X_{N}$ and $r_{6}[\lambda](1)\neq 0$, then
\begin{equation*}
\begin{aligned}
G_{M_{1}}[\lambda](t,s)=&G_{N}[\lambda](t,s)-\frac{r_{6}[\lambda](t)}{r_{6}[\lambda](1)}\, G_{N}[\lambda](1,s)\\
=&G_{N}[\lambda](t,s)-\frac{G_{N}[\lambda](t,1)}{G_{N}[\lambda](1,1)}\, G_{N}[\lambda](1,s),\;\; \forall (t,s)\in I\times I.
\end{aligned}
\end{equation*}
\end{theorem} 
We now do the process backwards by writing the Neumann problem based on the Mixed problems. 
Doing the calculations in a similar way we arrive at the next theorems. 
\begin{theorem}
Assume that $L[\lambda]$ is nonresonant in $X_{M_{1}}$ and $r'_{10}[\lambda](1)\neq 0$, then it holds that
\begin{equation*}
\begin{aligned}
G_{N}[\lambda](t,s)=&G_{M_{1}}[\lambda](t,s)-\frac{r_{10}[\lambda](t)}{r'_{10}[\lambda](1)}\, \frac{\partial }{\partial t} G_{M_{1}}[\lambda](1,s)\\
=&G_{M_{1}}[\lambda](t,s)-\dfrac{\frac{\partial }{\partial s}G_{M_{1}}[\lambda](t,1)}{ \frac{\partial^{2}}{\partial s\partial t}G_{M_{1}}[\lambda](1,1)}\, \frac{\partial }{\partial t} G_{M_{1}}[\lambda](1,s),\;\; \forall (t,s)\in I\times I.
\end{aligned}
\end{equation*}
\end{theorem}
\begin{theorem}
	Assume that $L[\lambda]$ is nonresonant in $X_{M_{2}}$ and $r_{7}[\lambda](0)\neq 0$, then it holds that 
\begin{equation*}
\begin{aligned}
G_{N}[\lambda](t,s)=&G_{M_{2}}[\lambda](t,s)-\frac{r_{7}[\lambda](t)}{r'_{7}[\lambda](0)}\, \frac{\partial}{\partial t}G_{M_{2}}[\lambda](0,s)\\
=&G_{M_{2}}[\lambda](t,s)-\dfrac{\frac{\partial }{\partial s}G_{M_{2}}[\lambda](t,0)}{\frac{\partial^{2}}{\partial s \partial t}G_{M_{2}}[\lambda](0,0)}\, \frac{\partial}{\partial t}G_{M_{2}}[\lambda](0,s),\;\; \forall (t,s)\in I\times I.
\end{aligned}
\end{equation*}
\end{theorem}
\subsection{Periodic and Neumann problems}
In this section we look for a relationship between the Green's functions $G_{P}[\lambda]$ and $G_{N}[\lambda]$ following the same steps than in previous sections.
\begin{theorem}
	Assume that $L[\lambda]$ is nonresonant in $X_{P}$ and $r'_{3}[\lambda](1)\neq 0$, then it holds that 
\begin{equation*}
\begin{aligned}
G_{N}[\lambda](t,s)=&G_{P}[\lambda](t,s)-\frac{r_{3}[\lambda](t)}{r'_{3}[\lambda](1)}\, \frac{\partial }{\partial t} G_{P}[\lambda](1,s)\\
=&G_{P}[\lambda](t,s)-\dfrac{\frac{\partial }{\partial s}G_{P}[\lambda](t,0)}{\frac{\partial^{2} }{\partial s\partial t}G_{P}[\lambda](1,0)}\, \frac{\partial }{\partial t} G_{P}[\lambda](1,s),\;\; \forall (t,s)\in I\times I.
\end{aligned}
\end{equation*}
\end{theorem}
Finally doing the reverse process by studying the Periodic problem as a function of the Neumann one, 
we deduce the following result. 
\begin{theorem}
If operator $L[\lambda]$ is nonresonant in $X_{N}$ and $$|I-A_{N}^{2}[\lambda]|:=r_{5}[\lambda](1)-r_{5}[\lambda](0)+r_{6}[\lambda](1)-r_{6}[\lambda](0)\neq 0,$$ then the next equality is fulfilled
\begin{equation*}
\footnotesize
\begin{aligned}
G_{P}[\lambda](t,s)=&G_{N}[\lambda](t,s)+\frac{1}{|I-A_{N}^{2}[\lambda]|}\, \Big(r_{5}[\lambda](t)+r_{6}[\lambda](t)\Big)\,\Big(G_{N}[\lambda](0,s)-G_{N}[\lambda](1,s)\Big)\\
=&G_{N}[\lambda](t,s)+\frac{1}{|I-A_{N}^{2}[\lambda]|}\, \Big(G_{N}[\lambda](t,0)-G_{N}[\lambda](t,1)\Big)\,\Big(G_{N}[\lambda](0,s)-G_{N}[\lambda](1,s)\Big),\;\; \forall (t,s) \in I\times I.
\end{aligned}
\end{equation*}
\end{theorem}
\subsection{Periodic and Mixed problems}
The same arguments of the previous subsections are applicable to the Periodic and Mixed problems.
\begin{theorem}
	If operator $L[\lambda]$ is nonresonant in $X_{P}$ and $$|I-A_{P}^{2}[\lambda]|:=\Big(1-r_{3}[\lambda](0)\Big)\Big(1+r'_{4}[\lambda](1)\Big)+r_{4}[\lambda](0)\, r'_{3}[\lambda](1)\neq 0,$$ then the next equality is fulfilled
\begin{equation*}
\small
\begin{aligned}
G_{M_{1}}[\lambda](t,s)=&G_{P}[\lambda](t,s)+\frac{r_{3}[\lambda](t)}{|I-A_{P}^{2}[\lambda]|}\, \left( \Big(1+r'_{4}[\lambda](1)\Big)\, G_{P}[\lambda](0,s)-r_{4}[\lambda](0)\, \frac{\partial}{\partial t} G_{P}[\lambda](1,s) \right)\\
&-\frac{r_{4}[\lambda](t)}{|I-A_{P}^{2}[\lambda]|}\, \left(r'_{3}[\lambda](1)\, G_{P}[\lambda](0,s)+\Big(1-r_{3}[\lambda](0)\Big)\, \frac{\partial}{\partial t} G_{P}[\lambda](1,s)\right)\\
=&G_{P}[\lambda](t,s)-\frac{1}{|I-A_{P}^{2}[\lambda]|}\, \left(1+\frac{\partial }{\partial t} G_{P}[\lambda](1,0)\right) \,\frac{\partial }{\partial s}G_{P}[\lambda](t,0)\, G_{P}[\lambda](0,s)\\
&+\frac{G_{P}[\lambda](0,0)}{|I-A_{P}^{2}[\lambda]|}\, \frac{\partial }{\partial s}G_{P}[\lambda](t,0)\, \frac{\partial}{\partial t} G_{P}[\lambda](1,s)\\
&+\frac{1}{|I-A_{P}^{2}[\lambda]|}\, \frac{\partial^{2} }{\partial s\partial t} G_{P}(1,0)\, G_{P}[\lambda](t,0)\, G_{P}[\lambda](0,s)\\
&-\frac{1}{|I-A_{P}^{2}[\lambda]|}\,\left(1+\frac{\partial }{\partial s} G_{P}[\lambda](0,0)\right)\, G_{P}[\lambda](t,0)\, \frac{\partial}{\partial t} G_{P}[\lambda](1,s),\;\; \forall (t,s) \in I\times I.
\end{aligned}
\end{equation*}
\end{theorem}
Similarly, we do a similar study of Mixed 2 problem  as a function of the Periodic problem.
\begin{theorem}
	If operator $L[\lambda]$ is nonresonant in $X_{P}$ and $$|I-A_{P}^{3}[\lambda]|=\Big(1+r_{3}[\lambda](1)\Big)\Big(1-r'_{4}[\lambda](0)\Big)+r'_{3}[\lambda](0)\, r_{4}[\lambda](1)\neq 0,$$ then the next equality is fulfilled
\begin{equation*}
\begin{aligned}
G_{M_{2}}[\lambda](t,s)=&G_{P}[\lambda](t,s)-\frac{r_{3}[\lambda](t)}{|I-A_{P}^{3}[\lambda]|}\, \left(\Big(1-r'_{4}[\lambda](0)\Big)\, G_{P}[\lambda](1,s)+r_{4}[\lambda](1)\, \frac{\partial}{\partial t} G_{P}[\lambda](0,s) \right)\\
&-\frac{r_{4}[\lambda](t)}{|I-A_{P}^{3}[\lambda]|}\,\left(r'_{3}[\lambda](0)\, G_{P}[\lambda](1,s)-\Big(1+r_{3}[\lambda](1)\Big)\, \frac{\partial}{\partial t} G_{P}[\lambda](0,s)\right)\\
=&G_{P}[\lambda](t,s)+\frac{1}{|I-A_{P}^{3}[\lambda]|}\,\left(1-\frac{\partial }{\partial t}G_{P}[\lambda](0,0)\right)\, \frac{\partial}{\partial s}G_{P}[\lambda](t,0)\, G_{P}[\lambda](1,s)\\
&+\frac{G_{P}[\lambda](1,0)}{|I-A_{P}^{3}[\lambda]|}\, \frac{\partial}{\partial s}G_{P}[\lambda](t,0)\, \frac{\partial}{\partial t} G_{P}[\lambda](0,s)\\
&+\frac{1}{|I-A_{P}^{3}[\lambda]|}\, \frac{\partial^{2}}{\partial s\partial t} G_{P}[\lambda](0,0)\,G_{P}[\lambda](t,0)\, G_{P}[\lambda](1,s)\\
&+\frac{1}{|I-A_{P}^{3}[\lambda]|}\,\left(1-\frac{\partial }{\partial s}G_{P}[\lambda](1,0)\right) G_{P}[\lambda](t,0)\, \frac{\partial}{\partial t} G_{P}[\lambda](0,s),\;\; \forall (t,s) \in I\times I.
\end{aligned}
\end{equation*}
\end{theorem}
Now we will do the process backwards by writing the Periodic problem based on the Mixed problems. 
Performing the calculations analogously to the previous subsections we deduce the next theorems. 
\begin{theorem}
Assume that $L[\lambda]$ is nonresonant in $X_{M_{2}}$ and $$|I-A_{M_{2}}|:=\Big(1-r_{7}[\lambda](1)\Big)\Big(1-r'_{8}(0)\Big)-r_{8}[\lambda](1)\,r'_{7}[\lambda](0)\neq 0,$$ then
\begin{equation*}
\begin{aligned}
G_{P}[\lambda](t,s)=&G_{M_{2}}[\lambda](t,s)+\frac{1}{|I-A_{M_{2}}|} \Big( \Big(1-r'_{8}[\lambda](0) \Big)\,r_{7}[\lambda](t)\, G_{M_{2}}[\lambda](1,s)\Big. \\
&+r_{8}[\lambda](1)\,r_{7}[\lambda](t) \frac{\partial}{\partial t}G_{M_{2}}[\lambda](0,s)+r'_{7}[\lambda](0)\,r_{8}[\lambda](t)\,G_{M_{2}}(1,s)\\
&\left.+\Big(1-r_{7}[\lambda](1)\Big)\,r_{8}[\lambda](t)\, \frac{\partial }{\partial t}G_{M_{2}}[\lambda](0,s)\right)\\
=&G_{M_{2}}[\lambda](t,s)+\frac{1}{|I-A_{M_{2}}|}\left( -\left(1+\frac{\partial }{\partial t}G_{M_{2}}[\lambda](0,1)\right)\,\frac{\partial }{\partial s}G_{M_{2}}[\lambda](t,0)\, G_{M_{2}}[\lambda](1,s) \right.\\
&+G_{M_{2}}[\lambda](1,1)\,\frac{\partial}{\partial s}G_{M_{2}}[\lambda](t,0) \frac{\partial}{\partial t}G_{M_{2}}[\lambda](0,s)\\
&+\frac{\partial^{2}}{\partial s\partial t}G_{M_{2}}[\lambda](0,0)\,G_{M_{2}}[\lambda](t,1)\,G_{M_{2}}(1,s)\\
&\left.-\left(1+\frac{\partial }{\partial s}G_{M_{2}}[\lambda](1,0)\right)\,G_{M_{2}}[\lambda](t,1)\, \frac{\partial }{\partial t}G_{M_{2}}[\lambda](0,s)\right),\;\; \forall (t,s) \in I\times I.
\end{aligned}
\end{equation*}
\end{theorem}
\begin{theorem}
Assume that $L[\lambda]$ is nonresonant in $X_{M_{1}}$ and $$|I-A_{M_{1}}[\lambda]|:=\Big(1-r'_{9}[\lambda](1)\Big)\Big(1-r_{10}[\lambda](0)\Big)-r_{9}[\lambda](0)\, r'_{10}[\lambda](1)\neq 0,$$ then 
\begin{equation*}
\small
\begin{aligned}
G_{P}[\lambda](t,s)=&G_{M_{1}}[\lambda](t,s)+\frac{r_{9}[\lambda](t)}{|I-A_{M_{1}}[\lambda]|}\left(\left(1-r_{10}[\lambda](0)\right)\,\frac{\partial }{\partial t}G_{M_{1}}[\lambda](1,s)+r'_{10}[\lambda](1)\, G_{M_{1}}[\lambda](0,s)\right)\\
&+\frac{r_{10}[\lambda](t)}{|I-A_{M_{1}}[\lambda]|}\left(r_{9}[\lambda](0)\,\frac{\partial }{\partial t}G_{M_{1}}[\lambda](1,s)+\left(1-r'_{9}[\lambda](1)\right)\, G_{M_{1}}[\lambda](0,s)\right)\\
=&G_{M_{1}}[\lambda](t,s)+\frac{1}{|I-A_{M_{1}}[\lambda]|}\,\left(1+\frac{\partial}{\partial s} G_{M_{1}}[\lambda](0,1)\right)\,G_{M_{1}}[\lambda](t,0)\,\frac{\partial }{\partial t}G_{M_{1}}[\lambda](1,s)\\
&-\frac{1}{|I-A_{M_{1}}[\lambda]|}\,\frac{\partial^{2}}{\partial s\partial t}G_{M_{1}}[\lambda](1,1)\,G_{M_{1}}[\lambda](t,0)\, G_{M_{1}}[\lambda](0,s)\\
&-\frac{G_{M_{1}}[\lambda](0,0)}{|I-A_{M_{1}}[\lambda]|}\,\frac{\partial}{\partial s}G_{M_{1}}[\lambda](t,1)\,\frac{\partial}{\partial t}G_{M_{1}}[\lambda](1,s)\\
&-\frac{1}{|I-A_{M_{1}}[\lambda]|}\,\left(1-\frac{\partial}{\partial t}G_{M_{1}}[\lambda](1,0)\right)\, \frac{\partial}{\partial s}G_{M_{1}}[\lambda](t,1)\,G_{M_{1}}[\lambda](0,s),\;\; \forall (t,s) \in I\times I.
\end{aligned}
\end{equation*}
\end{theorem}
\section{Nonlinear problem}
 In this section we will study the existence of solutions of the nonlinear problem 
\begin{equation}\label{nonlinear-problem}
\left\{
\begin{aligned}
L_{n}\,u(t)&=f(t,u(t)), \;\; \text{a.e.}\;\; t\in I,\\
B_{i}(u)&=\delta_{i}\, C_i(u),\quad i=1,\ldots,n,
\end{aligned}
\right.
\end{equation}
with 
\[L_n u(t):= u^{(n)}(t)+ a_1(t)\, u^{(n-1)}(t)+ \cdots + a_n(t)\,u(t) \]
the general $n$-th order linear operator.

The existence results will be deduced from Schaefer's fixed point theorem of integral operators defined in Banach spaces.

We will consider also the homogeneous particular
\begin{equation}\label{nonlinear-problemh}
\left\{
\begin{aligned}
L_{n}\,u(t)&=f(t,u(t)), \;\; \text{a.e.}\;\; t\in I,\\
B_{i}(u)&=0,\quad i=1,\ldots,n.
\end{aligned}
\right.
\end{equation}
We will assume that the nonlinear part of problem \eqref{nonlinear-problem} satisfies the following regularity conditions:
\begin{flalign*}
&(H_{1}) \quad  \text{For} \:\;n\geq 2, \text{the function}\hspace{2ex} f:I\times \R\rightarrow \R\hspace{2ex} \text{is a}\,L^{1}-\text{Carathéodory function, that is,}&
\end{flalign*}

\begin{itemize}
\item[-] $f(\cdot,x)$ is measurable for all $x\in \R$.
\item[-] $f(t,\cdot)$ is continuous for a.e. $t\in I$.
\item[-] For every $R>0$ there exists $\phi_{R}\in L^{1}(\R)$ such that 
$$\Big|f(t,x)\Big|\leq \phi_{R}(t),$$
for all $x\in [-R,R]$ and a.e. $t\in I$.
\end{itemize}

For $n=1$, the function  $f:I\times \R\rightarrow \R$ is $L^{\infty}-$ Carathéodory function, that is,
\begin{itemize}
	\item[-] $f(\cdot,x)$ is measurable for all $x\in \R$.
	\item[-] $f(t,\cdot)$ is continuous for a.e. $t\in I$.
	\item[-] For every $r>0$ there exists $h_{r}\in L^{\infty}(\R)$ such that 
	$$\Big|f(t,x)\Big|\leq h_{r}(t),$$
	for all $x\in [-r,r]$ and a.e. $t\in I$.
\end{itemize}
\begin{flalign*}
&(H_{2})\quad  \exists K\in L^{1}(I), K\geq 0\;\; \text{such that}&
\end{flalign*}
\begin{equation*}
|f(t,x)-f(t,y)|\le K(t)\,|x-y|,\;\; \text{for all } x,y\in \mathbb{R} \;\; \text{and } \;\; t\in I.
\end{equation*}
Let us define $X\equiv (C(I),\| \cdot \|_{\infty})$ the real Banach space endowed  with the supremum  norm 
\begin{equation*}
\| u\|_{\infty}=\sup_{t\in I} \;\;|u(t)|,\quad \text{for all}\;\;u\in X. 
\end{equation*}
We will denote by $u_{A}$ and $u_{B}$ the solutions of problems \eqref{nonlinear-problem} and \eqref{nonlinear-problemh} respectively. We know that these solutions are given by the following expressions 
\begin{equation*}
\begin{aligned}
u_{A}(t)&=\displaystyle \int_{0}^{1} G(t,s,\delta_{1},\ldots,\delta_{n})\,f(s,u_{A}(s))\,ds,\\
u_{B}(t)&=\displaystyle \int_{0}^{1} g(t,s)\,f(s,u_{B}(s))\,ds,
\end{aligned}
\end{equation*}
where $G$ and $g$ are the Green's functions related to the linear problems obtained from \eqref{nonlinear-problem} and \eqref{nonlinear-problemh}, respectively. In particular, for $n=2$, this problems are  \eqref{e-linear-delta} and \eqref{e-u} and, for $n\neq 2$, they are formulated in an analogous way, with obvious notations. Furthermore, they are linked by the generalization of formula \eqref{e-formula} to arbitrary order:
\begin{equation}\label{e-formula3}
	G(t,s,\delta_{1},\dots,\delta_n):=g(t,s)+ \sum_{i=1}^{n} \sum_{j=1}^{n} \delta_{i} \, b_{ij} \, \omega_{i}(t)  \, C_j(g(\cdot,s)),\quad t, \; s \in I,
\end{equation}
As we see, this formula is totally analogous to \eqref{e-formula}, with obvious notations, and its proof can be consulted in \cite{Cabada}.

Let us define 
\begin{equation*}
\begin{aligned}
K^{1}&=\max_{t\in I}\;\; \displaystyle \int_{0}^{1} |g(t,s)|\,K(s)\, ds,\\
K_{ij}^{2}&=\max_{t\in I}\;\; |\omega_{i}(t)| \, \displaystyle \int_{0}^{1} |C_{j}(g(\cdot,s))| \,K(s)\,ds,\quad \forall i,j=1,\ldots,n,\\
K_{ij}^{3}&=\max_{t\in I}\;\; |\omega_{i}(t)| \, \displaystyle \int_{0}^{1} |C_{j}(g(\cdot,s))\,f(s,0)| \,ds,\quad \forall i,j=1,\ldots,n,\\
P&=\max_{t\in I} \displaystyle \int_{0}^{1} |G(t,s,\delta_{1},\ldots,\delta_{n})| \,K(s)\,ds,\\
Q&=\max_{t\in I} \displaystyle \int_{0}^{1} |G(t,s,\delta_{1},\ldots,\delta_{n})\, f(s,0)|\,ds.
\end{aligned}
\end{equation*}
We assume that the following condition is fulfilled:
\begin{flalign*}
&(H_{3})\qquad K^{1}<1.&
\end{flalign*}
\begin{theorem}\label{thdir}
If conditions  $(H_{2})$ and $(H_{3})$ hold, then the following inequality is fulfilled
\begin{equation}\label{desigualdadinf}
\|u_{B}-u_{A}\|_{\infty}\le \frac{1}{1-K^{1}}\left(\sum_{i=1}^{n} \sum_{j=1}^{n} |\delta_{i} \, b_{ij}| K_{ij}^{2}\,\|u_{A}\|_{\infty}+ \sum_{i=1}^{n} \sum_{j=1}^{n} |\delta_{i} \,  b_{ij}| \,K_{ij}^{3} \right).
\end{equation}
\end{theorem}
\begin{proof}
	
By using \eqref{e-formula} we have that
\begin{equation*} 
\begin{aligned}
u_{B}(t)-u_{A}(t)=&\displaystyle \int_{0}^{1} g(t,s)\,f(s,u_{B}(s))\,ds-\displaystyle \int_{0}^{1} G(t,s,\delta_{1},\ldots,\delta_{n})\,f(s,u_{A}(s))\,ds\\
=&\displaystyle \int_{0}^{1} g(t,s)\,\left(f(s,u_{B}(s))-f(s,u_{A}(s))\right)\,ds\\
&-\sum_{i=1}^{n} \sum_{j=1}^{n} \delta_{i} \, b_{ij} \, \omega_{i}(t)  \,\displaystyle\int_{0}^{1}   C_j(g(\cdot,s))\left(f(s,u_{A}(s))-f(s,0)\right)\,ds\\
&-\sum_{i=1}^{n} \sum_{j=1}^{n} \delta_{i} \, b_{ij} \, \omega_{i}(t)  \,\displaystyle\int_{0}^{1}   C_j(g(\cdot,s))\,f(s,0)\,ds.
\end{aligned}
\end{equation*} 
Then, for all $t\in I$, from $(H_{2})$, we infer that 
\begin{equation*}
\begin{aligned}
|u_{B}(t)-u_{A}(t)|\le 
& \|u_{B}-u_{A}\|_{\infty}\,\displaystyle \int_{0}^{1} | g(t,s)| \,K(s)\,ds\\
+&  \|u_{A}\|_{\infty}\, \sum_{i=1}^{n} \sum_{j=1}^{n} | \delta_{i} \,  b_{ij}|  \, | \omega_{i}(t)|   \,\displaystyle\int_{0}^{1}   |C_j(g(\cdot,s))| \,K(s)\,ds\\
+&\sum_{i=1}^{n} \sum_{j=1}^{n} |\delta_{i} \, b_{ij}|  \, |\omega_{i}(t)| \,\displaystyle \int_{0}^{1}   |C_j(g(\cdot,s))\,f(s,0)| \,ds.
\end{aligned}
\end{equation*}	
Therefore 
\begin{equation*}
\|u_{B}-u_{A}\|_{\infty}\le 
K^{1}\, \|u_{B}-u_{A}\|_{\infty}+ \|u_{A}\|_{\infty}\, \sum_{i=1}^{n} \sum_{j=1}^{n} |\delta_{i}  \,  b_{ij}| \, 
K_{ij}^{2}+\sum_{i=1}^{n} \sum_{j=1}^{n} | \delta_{i} \,  b_{ij}| \, K_{ij}^{3},
\end{equation*}
that is, using $(H_{3})$,
\begin{equation*}
\|u_{B}-u_{A}\|_{\infty}\le \frac{1}{1-K^{1}}\left(\sum_{i=1}^{n} \sum_{j=1}^{n} |\delta_{i} \, b_{ij}| K_{ij}^{2}\,\|u_{A}\|_{\infty}+ \sum_{i=1}^{n} \sum_{j=1}^{n} |\delta_{i} \, b_{ij}| \,K_{ij}^{3} \right).
\end{equation*}
\end{proof}

\begin{corollary}\label{thdir1}
If conditions $(H_{2})$ and $(H_{3})$ hold, then the following inequalities are fulfilled
\begin{equation*}
\begin{aligned}
\|u_{B}\|_{\infty}&\le \frac{\displaystyle \sum_{i=1}^{n} \displaystyle \sum_{j=1}^{n} |\delta_{i} \, b_{ij}| K_{ij}^{2}-K^{1}+1}{1-K^{1}}\,\|u_{A}\|_{\infty}+\frac{\displaystyle \sum_{i=1}^{n} \displaystyle \sum_{j=1}^{n} |\delta_{i}\, b_{ij}| \,K_{ij}^{3}}{1-K^{1}},\\
\|u_{B}\|_{\infty}&\ge \frac{1-K^{1}-\displaystyle \sum_{i=1}^{n} \displaystyle \sum_{j=1}^{n} |\delta_{i} \, b_{ij}| K_{ij}^{2}}{1-K^{1}}\,\|u_{A}\|_{\infty}-\frac{\displaystyle \sum_{i=1}^{n} \displaystyle \sum_{j=1}^{n} |\delta_{i} \, b_{ij}| \,K_{ij}^{3}}{1-K^{1}}.
\end{aligned}
\end{equation*}
\end{corollary}
\begin{proof}
The proof is an immediate consequence of \eqref{desigualdadinf}  and the following inequality 
\begin{equation*}
\Big| \|u_{B}\|_{\infty}-\|u_{A}\|_{\infty}\Big|\le \|u_{B}-u_{A}\|_{\infty}. 
\end{equation*} 
\end{proof}

Next we will state Scheafer's fixed-point theorem (\cite{Tychnof}) that we will apply to the operator $T:X\rightarrow X$ given by
\begin{equation}\label{operador}
T\,u(t):=\displaystyle \int_{0}^{1} G(t,s,\delta_{1},\ldots,\delta_{n})\,f(s,u(s))\,ds,\;\; t\in I,
\end{equation}
 to guarantee the existence of a solution of problem \eqref{nonlinear-problem}.
 \begin{theorem}[Schaefer]\label{Schaefer}
 	Let $T:X \rightarrow X$ be a continuous and compact mapping of a Banach space $X$, such that the set 
 	\begin{equation*}
 	\{x\in X:\;\; x=\mu\,T\,x\;\; \text{for some}\;\; 0\le \mu\le 1\}
 	\end{equation*}
 	is bounded. Then $T$ has a fixed point.
 \end{theorem}
Now, we will  use Schaefer's theorem to ensure the existence of solutions of the nonlinear problem	\eqref{nonlinear-problem}.

\begin{theorem}\label{thx}
Assume that $(H_{1})$ and $(H_{2})$ hold and $P<1$. Then problem \eqref{nonlinear-problem} has at least one solution $u\in X$. 	
\end{theorem}

\begin{proof}
First, note that the fixed points of  operator $T$ defined in \eqref{operador} coincide with the solutions of  problem \eqref{nonlinear-problem}.

Now, we show that operator $T$ is compact. Since $G(t,s,\delta_{1},\ldots,\delta_{n})$ continuous and $f$ Carathéodory, we have that operator $T$ is continuous too.

Next, we will prove that $T$ maps bounded sets into relatively compact sets. Let $H\subset X$ be a bounded set. Since $H$ is bounded, there exists $r\in \mathbb{R}$, $r>0$ such that $\|u \|_{\infty}\le r$ for all $u\in H$. Then,
\begin{equation*}
\begin{aligned}
|T\, u(t)|&\le \displaystyle \int_{0}^{1} |G(t,s,\delta_{1},\ldots,\delta_{n})|| f(s,u(s))-f(s,0)|\,ds+\displaystyle \int_{0}^{1} |G(t,s,\delta_{1},\ldots,\delta_{n})| |f(s,0)|\,ds\\
&\le \|u\|_{\infty}\,\displaystyle \int_{0}^{1} |G(t,s,\delta_{1},\ldots,\delta_{n})|\,K(s)\,ds +\displaystyle \int_{0}^{1} |G(t,s,\delta_{1},\ldots,\delta_{n})||f(s,0)|\,ds.
\end{aligned}
\end{equation*}
So, for all $u\in H$, we have that 
 \begin{equation}\label{formula2}
 \| T\, u\|_{\infty} \le r\,P +Q,
\end{equation}
that is, $T(H)$ is bounded.

Let us show now the  equicontinuity of $T$.  For all  $t\in I$ and $u\in H$, we have that
$$ \begin{aligned}
| (T\, u) '(t) | &=  \left | \int_{0} ^ {1} \frac {\partial} {\partial t} G(t,s,\delta_{1},\ldots,\delta_{n})\, f (s, u (s))\, ds \right | \leq
\int_{0} ^ {1} \left | \frac {\partial} {\partial t} G(t,s,\delta_{1},\ldots,\delta_{n}) \right || f (s, u (s)) |\, ds   \\
&\leq \int_{0} ^ {1} \left | \frac {\partial} {\partial t} G(t,s,\delta_{1},\ldots,\delta_{n}) \right | \, \phi_{r}(s)\,ds.
 \end{aligned}$$
If $n\geq 2$, then the regularity of the Green's function $G(t,s,\delta_{1},\ldots,\delta_{n})$ allows us guarantee that it exists  $ M \in \R $, $M>0$ such that 
 $\left| \frac {\partial} {\partial t} G(t,s,\delta_{1},\ldots,\delta_{n}) \right | \le M$. Therefore, $$\int_{0} ^ {1} \left | \frac {\partial} {\partial t} G(t,s,\delta_{1},\ldots,\delta_{n}) \right | \, \phi_{r}(s)\,ds\leq M\,\int_{0} ^ {1} \phi_{r}(s)\,ds ..$$ So, for all $ t_1, t_2 \in I, \hspace {1ex} t_1 <t_2 $, we infer that 
$$ | (T\, u) (t_2)-(T\, u) (t_1) | = \left | \int_{t_1} ^ {t_2} (T\, u) '(s) ds \right | \leq \int_{t_1} ^ { t_2} | (T\, u) '(s) | ds \leq N (t_2-t_1). $$

If $n=1$, then the regularity of the Green's function $G(t,s,\delta_{1})$ allows us to ensure that   it exists $ \tilde{N} \in \R $, $\tilde{N}>0$ such that 
$\int_{0}^{1} |G(t,s,\delta_{1})|\,\phi_{r}(s)\,ds  \le \tilde{N}$. 
Therefore,
 $$\int_{0} ^ {1} \left | \frac {\partial} {\partial t} G(t,s,\delta_{1}) \right | \, \phi_{r}(s)\,ds=\int_{0} ^ {1} |a_{1}(t)||G(t,s,\delta_{1})| \, \phi_{r}(s)\,ds\leq \tilde{N}\, |a_{1}(t)|.  $$ 
 Then, for all $ t_1, t_2 \in I, \hspace {1ex} t_1 <t_2 $, we have  that 
$$ | (T\, u) (t_2)-(T\, u) (t_1) | = \left | \int_{t_1} ^ {t_2} (T\, u) '(s) ds \right | \leq \int_{t_1} ^ { t_2} | (T\, u) '(s) | ds \leq \tilde{N}\,\int_{t_{1}}^{t_{2}} |a_{1}(s)|\,ds. $$

Thus, $ T(H) $ is an equicontinuous set in $ X $. By Arzel\`{a}-Ascoli's Theorem, we deduce that $T(H)$ is relatively compact, that is, $T$ is a compact operator.
	
Let $u\in X$ be  such that $u=\mu\,T\,u$ for some $0\le \mu\le 1$. Then, using \eqref{formula2} we have that 
\begin{equation*}
\|u\|_{\infty}=\mu \|T\,u\|_{\infty}\le \|T\,u\|_{\infty}\le  \|u\|_{\infty}\,P +Q.
\end{equation*}
Thus
 \begin{equation*}
 \|u\|_{\infty}\le \frac{Q}{1-P}.
 \end{equation*}
 Therefore, applying Schaefer's Theorem, we conclude that problem \eqref{nonlinear-problem} has at least one solution $u\in X$. 
\end{proof}

\begin{remark}
We note that by definition of $X$, $(T\,u)$ is  not necessarily derivable. However, $(T\,u)'$ always exists because of the regularity of Green's function.
\end{remark}  

Next we will apply the above results to the particular case of the nonlinear second order Dirichlet problem. 

Suppose there exists $u_{D}$ a solution of the nonlinear Dirichlet problem
\begin{equation}\label{dir}
L[\lambda]\,u(t)=f(t,u(t)),\;\; \text{a.e.}\;\; t\in I,\quad u(0)=u(1)=0
\end{equation}
 and $u_{P}$ solution of the nonlinear Periodic problem
 \begin{equation*}
 L[\lambda]\,u(t)=f(t,u(t)),\;\; \text{a.e.}\;\; t\in I,\quad u(0)-u(1)=u'(0)-u'(1)=0.
 \end{equation*}

By definition of the Green's functions, we have that
\begin{equation*}
u_{D}(t)=\displaystyle \int_{0}^{1} G_{D}[\lambda](t,s)\,f(s,u_{D}(s))\,ds
\end{equation*}
and
\begin{equation*}
u_{P}(t)=\displaystyle \int_{0}^{1} G_{P}[\lambda](t,s)\,f(s,u_{P}(s))\,ds.
\end{equation*} 
We know from Theorem \ref{teorema2} that 
\begin{equation}\label{Dir}
G_{D}[\lambda](t,s)=G_{P}[\lambda](t,s)-\frac{G_{P}[\lambda](t,0)}{G_{P}[\lambda](1,0)}\, G_{P}[\lambda](1,s),\;\; \forall (t,s)\in I\times I.
\end{equation}

Let us define
 \begin{equation*}
 \begin{aligned}
 K_{1}&=\max_{t\in I}\;\; \displaystyle \int_{0}^{1} |G_{P}[\lambda](t,s)|\,K(s)\, ds,\\
 K_{2}&=\max_{t\in I}\;\;  \Big |\frac{G_{P}[\lambda](t,0)}{G_{P}[\lambda](1,0)} \Big|  \, \displaystyle \int_{0}^{1} |G_{P}[\lambda](1,s))|\,K(s)\,ds,\\
 K_{3}&=\max_{t\in I}\;\; \Big|\frac{G_{P}[\lambda](t,0)}{G_{P}[\lambda](1,0)}\Big| \, \displaystyle \int_{0}^{1} |G_{P}[\lambda](1,s)\,f(s,0)|\,ds ,\\
 P_{D}&=\max_{t\in I} \displaystyle \int_{0}^{1} |G_{D}[\lambda](t,s)| \,K(s)\,ds,\\
 Q_{D}&=\max_{t\in I} \displaystyle \int_{0}^{1} |G_{D}[\lambda](t,s)\, f(s,0)|\,ds.
 \end{aligned}
 \end{equation*}
 
As a direct consequence of Theorem \ref{thdir} and Corollary \ref{thdir1} we arrive at the follow results: 
 \begin{theorem}\label{th}
 	Suppose that $(H_{2})$ holds and $K_{1}<1$, then the following inequality is fulfilled
 	\begin{equation}\label{thxxx}
 	\|u_{D}-u_{P}\|_{\infty}\le \frac{1}{1-K_{1}}\left(K_{2}\,\|u_{D}\|_{\infty}+K_{3}\right).
 	\end{equation}
 \end{theorem}
 \begin{corollary}
 	Assume that $(H_{2})$ holds and $K_{1}<1$. Then, the following inequalities are fulfilled
 \begin{equation*}
 \begin{aligned}
 \|u_{P}\|_{\infty}&\le \frac{K_{2}-K_{1}+1}{1-K_{1}}\,\|u_{D}\|_{\infty}+\frac{K_{3}}{1-K_{1}},\\
 \|u_{P}\|_{\infty}&\ge \frac{1-K_{1}-K_{2}}{1-K_{1}}\,\|u_{D}\|_{\infty}-\frac{K_{3}}{1-K_{1}}.
 \end{aligned}
 \end{equation*}
\end{corollary}
 	
 \begin{theorem}\label{thxx}
 	Assume that $(H_{1})$ and $(H_{2})$ hold and $P_{D}<1$. Then the Dirichlet problem \eqref{dir} has at least one solution. 	
 \end{theorem}
\begin{remark}
 The same previous arguments can be applied to the rest of the problems discussed in this article using the formulas that relate Green's functions obtained in the previous section.
\end{remark}
In the sequel, we present an example to illustrate our results. 
\begin{example}
Consider the following equation	
\begin{equation*}
u''(t)-\,u(t)= \frac{c}{\sqrt{t}}\,e^{-u^{2}(t)},\;\; \text{a.e.} \;\; t\in I, \;\;\text{and}\;\; c>0.
\end{equation*}	

In this case, $f(t,u)= \frac{c}{\sqrt{t}}\, e^{-u^{2}}$ is $L^{1}$-Carathéodory function  and $f(t,0)= \frac{c}{ \sqrt{t}}\neq 0$ for all $t \in (0,1]$. Moreover, it is immediate to see that $f$ satisfy the condition $(H_{2})$ with $K(t)=c\,\sqrt{\frac{2}{e \,t}}$ for a.e. $t \in [0,1]$.

We have that the Green's function of the periodic problem is given by 
\begin{equation*}
G_{P}(t,s)=\left\{
\begin{aligned}
&\dfrac{e^{s-t+1}+e^{t-s}}{2(1- e)},\;\; 0\leq s\leq t\leq 1,\\
	&\dfrac{e^{t-s+1}+e^{s-t}}{2(1-e)},\;\; 0\leq t< s\leq 1,
	\end{aligned}
\right.
\end{equation*} 
and that of the Dirichlet problem is
\begin{equation*}
G_{D}(t,s)=\left\{
\begin{aligned}
&-\dfrac{(e^{2 s}-1)(e^{2}-e^{2t}) e^{-(s+t)}}{2(e^{2}-1)},\;\; 0\leq s\leq t\leq 1,\\
&\dfrac{(e^{2(s-1)}-1)(e^{2 t}-1) e^{-(s+t-2)}}{2(e^{2}-1)},\;\; 0\leq t< s\leq 1.
\end{aligned}
\right.
\end{equation*}


With the notation used in Theorem \ref{th} and by numerical approach, it can be seen that  
\begin{equation*}
\begin{aligned}
K_{1}&=\max_{t\in I}\;\; \displaystyle \int_{0}^{1} |G_{P}(t,s)|\,K(s) \, ds\approx 1.7472\,c,\\
K_{2}&=\max_{t\in I}\;\;   \Big |\frac{G_{P}(t,0)}{G_{P}(1,0)}\Big|  \, \displaystyle \int_{0}^{1} |G_{P}(1,s))|\,K(s)\,ds\approx 1.744\,c,\\
K_{3}&=\max_{t\in I}\;\; \Big|\frac{G_{P}(t,0)}{G_{P}(1,0)} \, \displaystyle \int_{0}^{1} G_{P}[\lambda](1,s)\,f(s,0)\,ds\Big|\approx 2.033\, c,\\
P_{P}&=\max_{t\in I} \displaystyle \int_{0}^{1} |G_{P}(t,s)|\,K(s) \,ds\approx 1.7472\,c,\\
Q_{P}&=\max_{t\in I} \displaystyle \int_{0}^{1} |G_{D}(t,s)\, f(s,0)|\,ds\approx 2.0369\,c,\\
P_{D}&=\max_{t\in I} \displaystyle \int_{0}^{1} |G_{D}(t,s)|\,K(s) \,ds\approx 0.1651\,c,\\
Q_{D}&=\max_{t\in I} \displaystyle \int_{0}^{1} |G_{D}(t,s)\, f(s,0)|\,ds\approx 0.179\,c.
\end{aligned}
\end{equation*}
Then, the conditions $K_{1}<1$ and $P_{D}<1$ are fulfilled if and only if 
$$0<c<\min\Big\{\frac{1}{0.1651},\frac{1}{1.7472} \Big\}\approx 0.572344.$$
Therefore if $0<c<0.572344$, then by Theorem \ref{thxx} there is at least one solution $u_{D}$ of  the Dirichlet problem
\begin{equation*}
u''(t)-\,u(t)=\frac{c}{\sqrt{t}}\,e^{-u^{2}(t)},\;\; \text{a.e.}\;\;t\in I,\;\;u(0)=u(1)=0.
\end{equation*} 
By the proof of Theorem \ref{thx} we have that 
\begin{equation*}
\|u_{P}\|_{\infty}\le \frac{Q_{P}}{1-P_{P}}\approx \frac{2.0369\,c}{1-1.7472\,c}
\end{equation*}
and
\begin{equation*}
\|u_{D}\|_{\infty}\le \frac{Q_{D}}{1-P_{D}}\approx \frac{0.179\,c}{1-0.1651\,c}.
\end{equation*}
As a consequence, we deduce that 
\begin{equation}\label{equation1}
 \begin{aligned} 
 	\|u_{P}-u_{D}\|_{\infty}&\leq \|u_{P}\|_{\infty}+\|u_{D}\|_{\infty}\leq \frac{2.0369\,c}{1-1.7472\,c}+ \frac{0.179\,c}{1-0.1651\,c}\\
&=\dfrac{c\,(7.68176-2.25\,c) }{c^{2}-6.62928\,c+3.46665}:=\gamma(c).
\end{aligned}
\end{equation}
On the other hand, if $0<c<0.572344$ applying inequality \eqref{thxxx} we obtain the following estimate of the distance between the solutions:
\begin{equation}\label{equation2}
\begin{aligned}
\|u_{P}-u_{D}\|_{\infty}&\le \frac{1}{1-K_{1}}\,\left(K_{2}\,\|u_{D}\|_{\infty}+K_{3} \right)\\
&\le\frac{1}{1-K_{1}}\,\left(K_{2}\,\frac{0.179\,c}{1-0.1651\,c}\,+K_{3} \right)\\
&\approx \frac{1}{1-1.7472\,c}\,\left(1.744\,c \,\frac{0.179\,c}{1-0.1651\,c}\,+2.033\,c\right) \\
&=\dfrac{c\,(7.0477-0.0813703\,c) }{c^{2}-6.62928\,c+3.46665}:=\psi(c).
\end{aligned}
\end{equation}
Comparing \eqref{equation1} and \eqref{equation2} (see Figure \ref{diferencia}), we have that the estimate \eqref{equation2} is better than \eqref{equation1} for $0<c<0.2878$ and worse for $0.2878<c<0.572344$.
\begin{figure}
\begin{center}
\includegraphics[width=9cm]{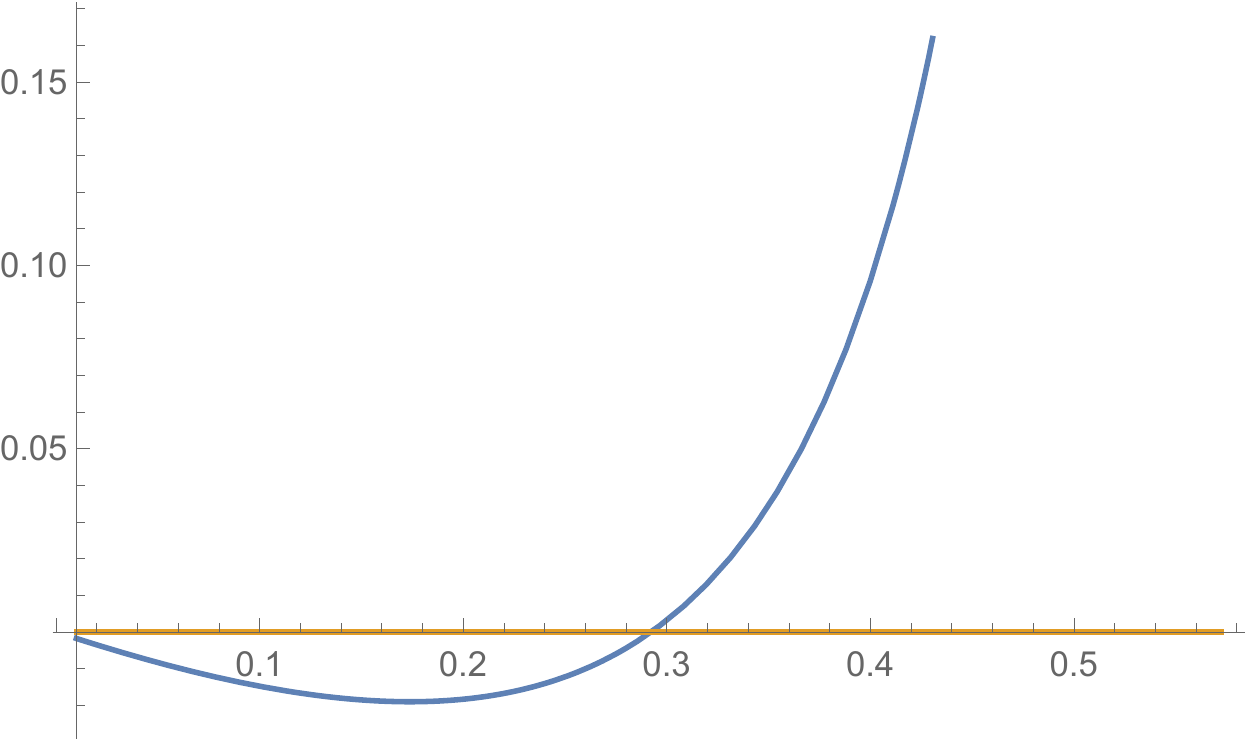}
	\caption{Representation of the function $\psi-\gamma$ on the interval $(0,0.572344)$.}
	\label{diferencia}
\end{center}
\end{figure}
\end{example}

\end{document}